\documentclass[times,twocolumn,final]{elsarticle}
\usepackage{framed,multirow}

\usepackage{amsmath}
\usepackage{amssymb}
\usepackage{amsthm}
\usepackage{latexsym}
\usepackage{url, hyperref}
\usepackage{graphicx}
\usepackage{verbatim}
\usepackage{algorithm2e}
\usepackage{caption}

\usepackage{floatrow}
\newfloatcommand{capbtabbox}{table}[][\FBwidth]

\usepackage{url}
\usepackage{xcolor}
\definecolor{newcolor}{rgb}{.8,.349,.1}

\usepackage{hyperref}

\newtheorem{dfn}[algocf]{Definition}
\newtheorem{pro}[algocf]{Problem}
\newtheorem{thm}[algocf]{Theorem}
\newtheorem{lem}[algocf]{Lemma}

\usepackage{enumitem}
\newlist{steps}{enumerate}{1}
\setlist[steps, 1]{label = Step \arabic*:}

\newcommand{\R}{\mathbb{R}}
\newcommand{\Z}{\mathbb{Z}}
\newcommand{\N}{\mathbb{N}}
\newcommand{\ga}{\gamma}
\newcommand{\Ga}{\Gamma}
\newcommand{\de}{\delta}
\newcommand{\ep}{\varepsilon}
\newcommand{\la}{\lambda}
\newcommand{\id}{\mathrm{id}}
\newcommand{\bs}{\hfill $\blacksquare$}

\usepackage[switch,pagewise]{lineno} 

\journal{Computers \& Graphics}

\begin{document}


\begin{frontmatter}

\title{Encoding and Topological Computation on Textile Structures}

\author[1]{Matthew Bright}{\corref{cor1}}
\cortext[cor1]{Corresponding author:}
\emailauthor{M.J.Bright@liverpool.ac.uk}{}
    
\author[2]{Vitaliy Kurlin}

\address[1,2]{Computer Science department and Materials Innovation Factory, University of Liverpool, Liverpool L68 3BX, UK}



\begin{abstract}
A textile structure is a periodic arrangement of threads in the thickened plane.
A topological classification of textile structures is harder than for classical knots and links that are non-periodic and restricted to a bounded region.
The first important problem is to encode all textile structures in a simple combinatorial way. 
This paper extends the notion of the \emph{Gauss code} in classical knot theory, providing a tool for topological computation on these structures. 
As a first application, we present a linear time algorithm for determining whether a code represents a textile in the physical sense. 
This algorithm, along with invariants of textile structures, allowed us for the first time to classify all oriented textile structures woven from a single component up to complexity five. 
\end{abstract}

\end{frontmatter}


\section{Introduction: motivations and the realizability problem}
\label{sec:intro}

We consider a textile structure as a set of continuous curves embedded without intersections in $\R^3$ so that any textile is periodically repeated in two directions, see Definition~\ref{dfn:textile}.
\smallskip

A textile struture can also be defined by an embedding of a finite number of closed curves (homeomorphic to a circle $S^1$) into a \emph{thickened torus} $T^2\times I$ as shown in Fig.~\ref{fig:diagonal_textile}, which is the product space of a 2-dimensional torus $T^2$ and an interval $I$.
Similarly, a classical \emph{knot} can be an embedding of a circle $S^1$ into a thickened plane $\R^2\times I$ or a thickened sphere $S^2\times I$.
\smallskip

A topological equivalence of knots and \emph{links} (embeddings of several disjoint curves) is an ambient \emph{isotopy}, a continuous deformation of the ambient space, see Definition~\ref{dfn:isotopy}.
This concept of an isotopy connects the investigation of textile structures to that of links in orientable thickened surfaces of higher genus, see ~\cite{Grishanov07},~\cite{grishanov2009topological1} and~\cite{grishanov2009topological2} for an exposition of this approach.
\smallskip

Computational tools for classical knot theory are now widely available, see~\cite{linknot}.
Structures in higher genus thickened surfaces have been less well explored. 
To date, there has been a manual enumeration of non-oriented knots with up to 5 crossings in the thickened torus were manually enumerated~\cite{akimova2014classification}.
\smallskip

\begin{figure}[htb]
\begin{center}
\includegraphics[width=\textwidth]{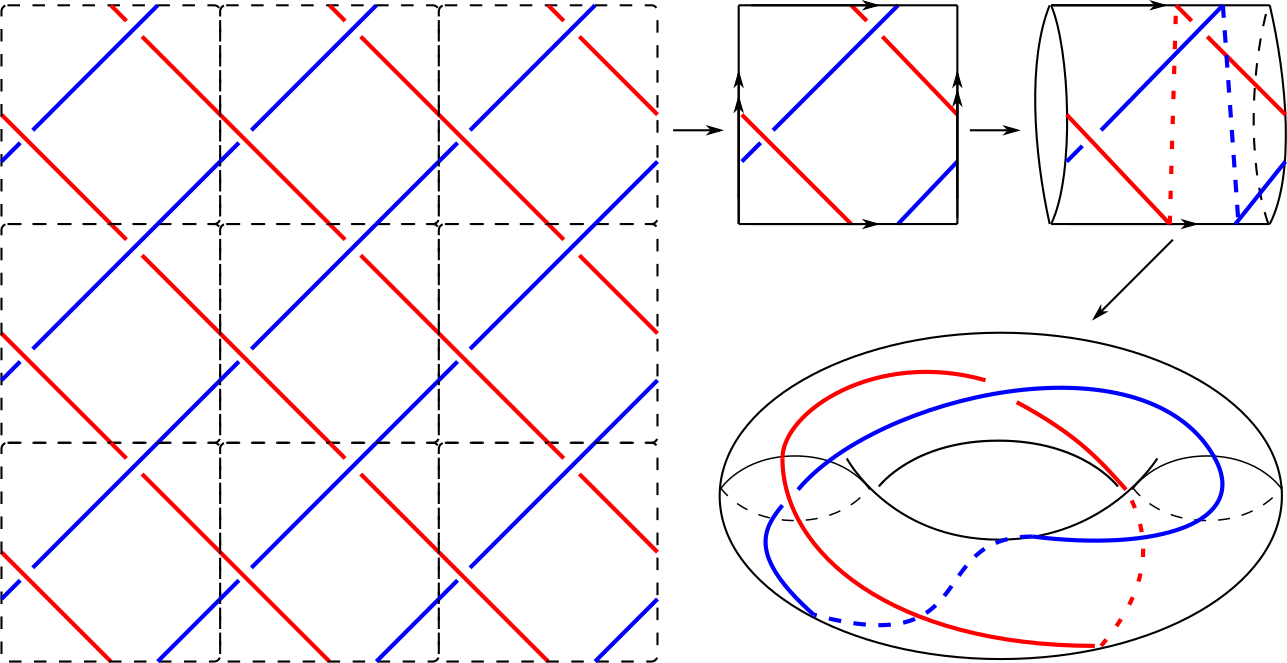}
\caption{A textile structure depicted as periodic crossings in a square which (after gluing opposite sides) can be represented as curves in a thickened torus.} 
\label{fig:diagonal_textile}
\end{center}
\end{figure}

Computation on knotted structures requires an approach to encoding them - for example, \emph{Gauss codes} are 1-dimensional string of symbols that encode all links in $\R^3$, see Definition~\ref{dfn:Gauss_codes}.
\smallskip
 
The last picture in Fig.~\ref{fig:Gauss_codes} shows that not every such code gives rise to a link in $\R^3$. 
\cite{kurlin2008Gauss} has described an algorithm for determining which Gauss codes are realizable by real links. 
This paper solves the following harder problem for textile structures.

\begin{pro}[encoding of textiles and realizability]
\label{pro:realizability} 
Encode any textile structure in such a way that allows an automatic enumeration by efficiently checking whether any potential code is realizable by a link embedded in the thickened torus $T^2\times I$.
\end{pro}

\begin{figure}[htb]
\begin{center}
\includegraphics[width=\textwidth]{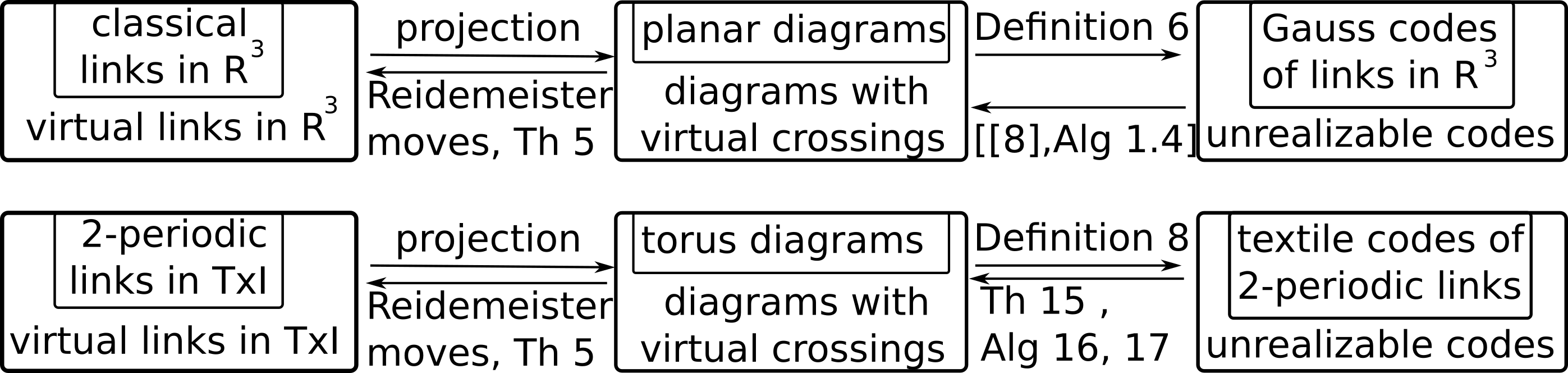}
\caption{Reducing dimensionality without losing information: from links in 3D to 2D diagrams, then to 1D codes, which can be efficiently manipulated.}
\label{fig:importance}
\end{center}
\end{figure}

Other approaches to knotted structures are 3-page embeddings that encode isotopy classes of spatial graphs by central elements of finitely presented semigroups \cite{kurlin2001dynnikov}, \cite{kurlin2004three}, \cite{kurlin2007three}, \cite{kurlin2015linear}.

Here are our contributions to the classification of textiles.
6

\noindent
$\bullet$
Definition~\ref{dfn:textile_code} introduces a new \emph{textile code} of any knot or link embedded in the thickened torus $T^2\times I$.

\noindent
$\bullet$
Theorem~\ref{thm:algorithm} proves that the linear time Algorithm~\ref{alg:periodB} detect realizability of any textile code by a link in $T^2\times I$.  

\noindent
$\bullet$
Tables~\ref{tab:zenkploy22},~\ref{tab:zenkploy2} and~\ref{tab:zenkploy3} obtained by Algorithm~\ref{alg:periodB} enumerate all oriented proper knots $T^2\times I$ up to complexity 5.

\section{Diagrams and Gauss codes of classical and virtual links}
\label{sec:basic}

A \emph{homemorphism} is a continuous bijection whose inverse is also continuous.
An \emph{embedding} is a continuous injective map. We consider a thickened surface $S\times I$, where $I\approx[0,1]$, is an interval.
When $S=\R^2$ or $S=S^2$ is a 2-dimensional sphere, we get classical knots.
When $S=T^2$ is a 2-dimensional torus (a product of circles $S^1\times S^1$), knots can be called \emph{2-periodic}. 
 
\begin{dfn}[knots and links]
\label{dfn:link}
A \emph{link} is an embedding $f:\sqcup_{i=1}^k S_i^1\to S\times I$ of several circles called \emph{components} so that all images $f(S_i^1)\subset S\times I$ are disjoint, i.e. have no (self-) intersections.
A link with a single component is called a \emph{knot}. 
If an orientation is imposed on the embedded circles, then the link itself is called \emph{oriented}, otherwise \emph{unoriented}. 
\bs
\end{dfn}

\begin{dfn}[isotopy]
\label{dfn:isotopy}
An ambient \emph{isotopy} between links $L_0,L_1$ in a thickened surface is a continuous family of homeomorphisms $f_t:S\times I\to S\times I$, $t\in[0,1]$, such that $f_0=\id$ is the identity and the final homemorphism $f_1$ takes $L_0$ to $L_1$. 
\bs
\end{dfn} 

Isotopy is the standard equivalence relation on knots and links. 
Links are depicted via their projections to a surface. 
 
\begin{dfn}
\label{dfn:diagram}
The \emph{diagram} of a link $L\subset S\times I$ in a thickened surface is the image of $L$ under the projection to $S$.
For $S=\R^2$, the diagram is called \emph{planar}. 
In a general position (after a small enough perturbation of $L$), the diagram consists of smooth arcs and transversal intersections called \emph{crossings}, see Fig.~\ref{fig:Gauss_codes}. 
\bs
\end{dfn} 

At each crossing we specify which arc of a link $L$ goes over (through an \emph{overcrossing} of this arc) another arc that passes through an \emph{undercrossing}.
An crossing in a diagram is  represented by a continuous arc (the overcrossing) passing between the ends of two disjoint arcs (the undercrossing). Crossing are \emph{signed} based on the direction of rotation between the undercrossing and overcrossing element, as shown in Fig.~\ref{fig:Gauss_codes}, right

\begin{thm}[Reidemeister \cite{reidemeister1927elementare}]
\label{thm:Reidemeister}
Links $L_0,L_1$ are isotopic in $\R^3$ if only if their planar diagrams can be obtained from each other by an isotopy of diagrams in $\R^2$ and finitely many Reidemeister moves in Fig.~\ref{fig:Reidemeister} (and all their symmetric images).
\end{thm}

\begin{figure}[htb]
\begin{center}
\includegraphics[width=\textwidth]{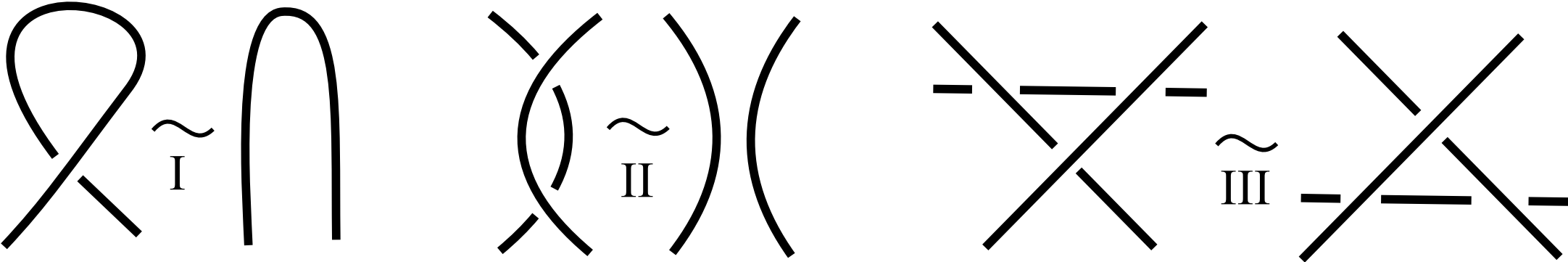}
\caption{Reidemeister moves generate isotopies of links in $\R^3$, see Theorem~\ref{thm:Reidemeister} RI resolves a self-intersection and RII separates two strands 'lying over each other' in the ambient space - both reduce the number of crossings. In RIII, a strand passes over (or under) a crossing - the number and orientation of crossings do not change.} 
\label{fig:Reidemeister}
\end{center}
\end{figure}


\begin{dfn}[Gauss codes]
\label{dfn:Gauss_codes}
Given an oriented link $L\subset\R^3$ with a planar diagram $D\subset\R^2$, arbitrarily label all double crossings by integers $1,\dots,n$.
For every component of $L$, trace the corresponding curve in the diagram $D$ and write down the indices of crossings one by one, starting from any crossing so that each word is considered up to cyclic permutations.
Every undercrossing symbol has a superscript equal to the sign defined in Fig.~\ref{fig:Gauss_codes}, while overcrossing symbols are unadorned.
\smallskip

An \emph{abstract} Gauss code is a set $W$ of words, comprised of symbols $i$ and $i^{\ep}$, where $i=1,\dots,n$, $\ep\in \{+, -\}$ such that each symbol $i$ appears only once, and for each $i$ exactly one of $\{i^+, i^-\}$ also appears in $W$. 
\bs
\end{dfn}

\newcommand{\h}{19mm}
\begin{figure}[htb]
\begin{center}
\includegraphics[height=\h]{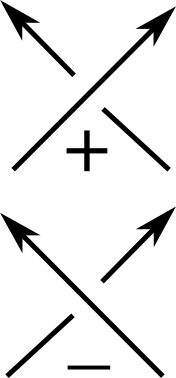}
\hspace*{2mm}
\includegraphics[height=\h]{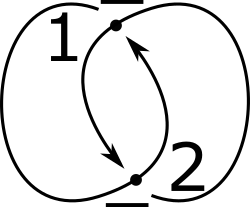}
\hspace*{2mm}
\includegraphics[height=\h]{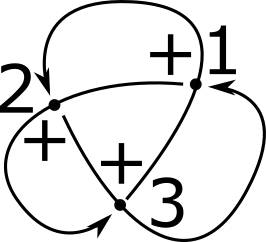}
\hspace*{2mm}
\includegraphics[height=\h]{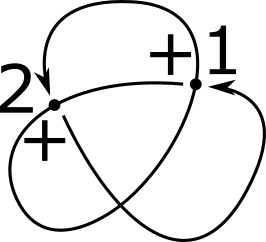}
\caption{\emph{1st}: signs of crossings. 
\emph{2nd}: the diagram of a Hopf link has the Gauss code $\{12^-,1^-2\}$.
\emph{3rd}: the diagram of a trefoil has the Gauss code $1^+ 2 3^+ 1 2^+ 3$.
\emph{4th}: any planar diagram with the code $1^+ 2 1 2^+$ includes a virtual crossing.
} 
\label{fig:Gauss_codes}
\end{center}
\end{figure}

Gauss codes were called paragraphs in~\cite{kurlin2008Gauss} for multi-component links in $\R^3$.
\smallskip 

An abstract Gauss code in the sense of Definition~\ref{dfn:Gauss_codes}, e.g. the code $1^+ 2 1 2^+$, may not represent a diagram of a real link. 
In the last picture of Fig.~\ref{fig:Gauss_codes} we have attempted to draw a diagram by joining the two positive crossings given by the code $1^+ 2 1 2^+$, which forces an extra intersection called a \emph{virtual} crossing without a specified overcrossing or undercrossing arc.
\smallskip

A \emph{virtual} link is a class of planar diagrams with virtual crossings considered up to Reidemeister moves similar to the classical ones in Fig.~\ref{fig:Reidemeister}, where any crossings can be virtual. 
\smallskip

To use Gauss codes to effect computations on links in $\R^3$, we need to work only with realizable codes. Kurlin~\cite{kurlin2008Gauss} has developed a linear time algorithm to determine whether an abstract Gauss code is realizable by a classical link in $\R^3$. 
\smallskip

First, we note an important connection between the theory of virtual links and those embedded in a thickened surface. 
There is a presentation of virtual links developed by Kamada and Saito~\cite{carterkamada}, in which virtual crossings are 'resolved' by embedding the resulting diagram into a surface with an additional $2$-handle (and thus higher genus) as depicted in Fig.~\ref{fig:virttorus}.
\smallskip

\begin{figure}[htb]
\begin{center}
\includegraphics[width=\textwidth]{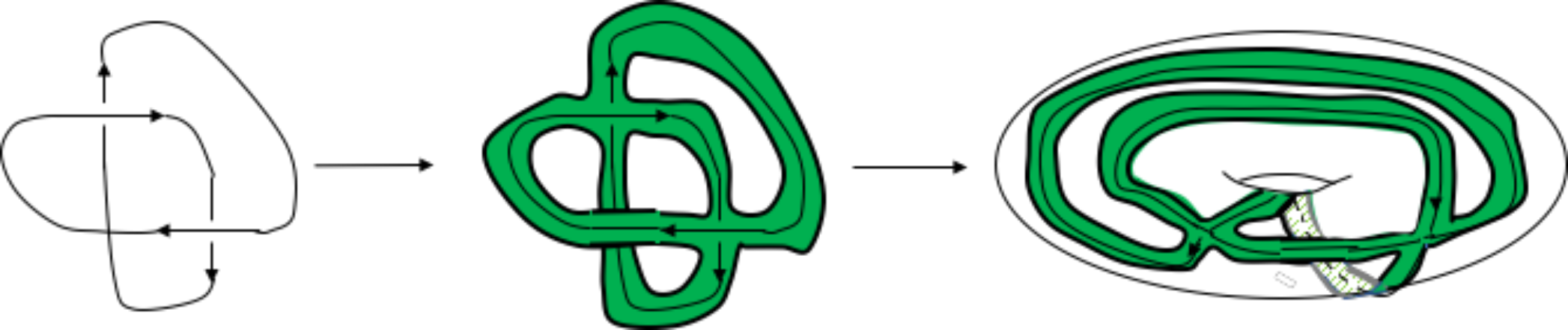}
\caption{The virtual knot with one virtual crossing of Fig.~\ref{fig:Gauss_codes} can be realized by a diagram on the surface of a torus with only classical crossings.}
\label{fig:virttorus}
\end{center}
\end{figure}

Kuperberg~\cite{kuperbergvirt05} has used the above representation to show that links in a thickened surface $S\times I$ uniquely correspond to virtual links, and that if two such links are equivalent as virtual links, then they are equivalent by an ambient isotopy in $S\times I$.
\smallskip

Theorem~\ref{thm:Reidemeister} applies equally to virtual link diagrams - indeed, there are additional 'virtual' moves which may be applied (see for example~\cite{kauffmannvirt}. 
Hence we may apply the Reidemeister moves in Fig.~\ref{fig:Reidemeister} to diagrams representing links in a thickened torus. 

\section{Textile codes of links in a thickened torus}
\label{sec:textile_codes}

This section introduces new textile codes that will allow us to systematically classify links up to ambient isotopy in a thickened torus.
First we clarify a difference between periodic textile structures and links in a fixed thickened torus $T^2\times I$.
\smallskip

After choosing any basis $\vec e_x,\vec e_y$ in $\R^2$, the plane can be mapped to the torus via the covering map $\ga:\R^2\to T^2=S^1\times S^1$ by sending $\vec e_x,\vec e_y$ to the meridian and longitude of the torus $T^2$.
In coordinates, any point $(x,y)\in\R^2$ is mapped to $(\{x\},\{y\})\in S^1\times S^1$.
Here $\{x\}=x-[x]\in[0,1)$ denotes the fractional part of any real coordinate $x\in\R$ and parameterizes the first factor circle $S^1$ of $T^2$, similarly for the coordinate $y$.
\smallskip

Conversely, any fixed torus $T^2$ with a meridian $\mu$ and longitude $\la$ can be obtained as the image $\ga(\R^2)$ whose basis vectors map to $\mu,\la$. 
The preimage $(\ga\times\id_I)^{-1}(L)$ of any link $L\subset T^2\times I$ is an infinite link $L\subset\R^2\times I$ preserved under the translations along the vectors that map to $\mu,\la$ and will be called a \emph{textile} not to confuse it with links in a fixed thickened torus $T^2\times I$. 
\smallskip

A torus $T^2$ can be obtained from a square by identifying its opposite sides with the same orientations, see Fig.~\ref{fig:diagonal_textile}.
 
\begin{dfn}[textiles, proper textiles]
\label{dfn:textile}
A \emph{textile} is an embedding of infinitely many lines or circles $L\subset\R^2\times I$ preserved under translations by two linearly independent vectors in $\R^2$ that form a \emph{basis} of $L$.
An \emph{equivalence} between textiles is any isotopy in $\R^2\times I$, not necessarily preserving their bases in $\R^2$.
A textile $L$ is \emph{proper} if there is no embedding $f: \R\times I \to (\R^2\times I)-L$ such that $f: \R \times \{+1, -1\} \subset \R^2\times \{+1, -1\}$:   see dashed projections of such strips in Fig.~\ref{fig:improper_textiles}. 
For a fixed covering $\ga:\R^2\times I\to T^2\times I$, a \emph{torus diagram} of a textile $L\subset\R^2\times I$ is the projection of $\ga(L)$ on $T^2$, which can be considered as a square with identified opposite sides and usual double crossings. The planar representation of textile components in the diagram is subject to the same restrictions as for knot projections, with the additional \textbf{periodic boundary condition} that arcs intersecting opposite edges of a diagram must coincide.
\bs
\end{dfn}

\begin{figure}[htb]
\begin{center}
\includegraphics[width=\textwidth]{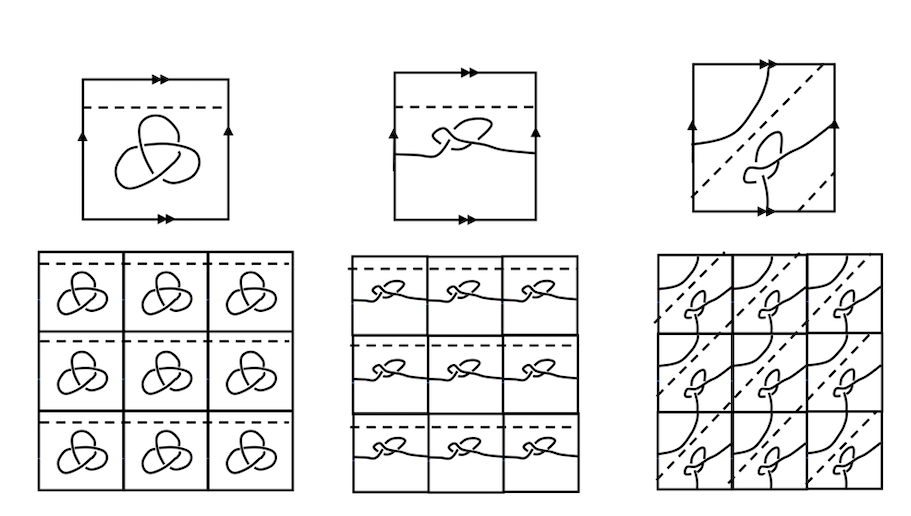}
\caption{Improper textiles have components separated by embedded strips.} 
\label{fig:improper_textiles}
\end{center}
\end{figure}

Textiles in Definition~\ref{dfn:textile} (called doubly periodic links in \cite{Morton})  are considered up to wider equivalences than isotopies in a fixed thickened torus. 
\smallskip

An isotopy between textiles $L_0,L_1$ is more general than an isotopy in a fixed thickened torus $T^2\times I$.
Indeed, $T^2$ can be homeomorphically mapped to itself by Dehn twists along its meridian or longitude, which keeps an isotopic class of a textile in $\R^2\times I$, but changes the isotopy type of a link in $T^2\times I$. 
\smallskip

Definition~\ref{dfn:Gauss_codes} of Gauss codes is extended to torus diagrams.

\begin{dfn}[textile codes]
\label{dfn:textile_code}
Let $L$ be a link in a thickened torus $T^2\times I$ and let $D\subset T^2$ be its torus diagram. 
In addition to labelling crossings by $1,\dots,n$ as in Definition~\ref{dfn:Gauss_codes}, we label the intersections of $D$ with a meridian (represented by the top and bottom edges of the square) with symbols $h_1,\dots, h_l$ from left to right and intersections with a longitude (the left and right edges of the diagram) with symbols $v_1,\dots, v_m$ from bottom to top. 
Then we trace each closed curve of $D$, writing the symbols as they are encountered. 
For an undercrossing $i$, add its sign as a superscript $\pm$. 
For every symbol $h_j, v_k$, add a superscript $\pm$ as in the left picture of Fig.~\ref{fig:signs_hv}.
The resulting collection $W$ of cyclic words is called a \emph{textile code} of the diagram $D$.
\bs
\end{dfn}

\begin{figure}[htb]
\centering
\includegraphics[height=27mm]{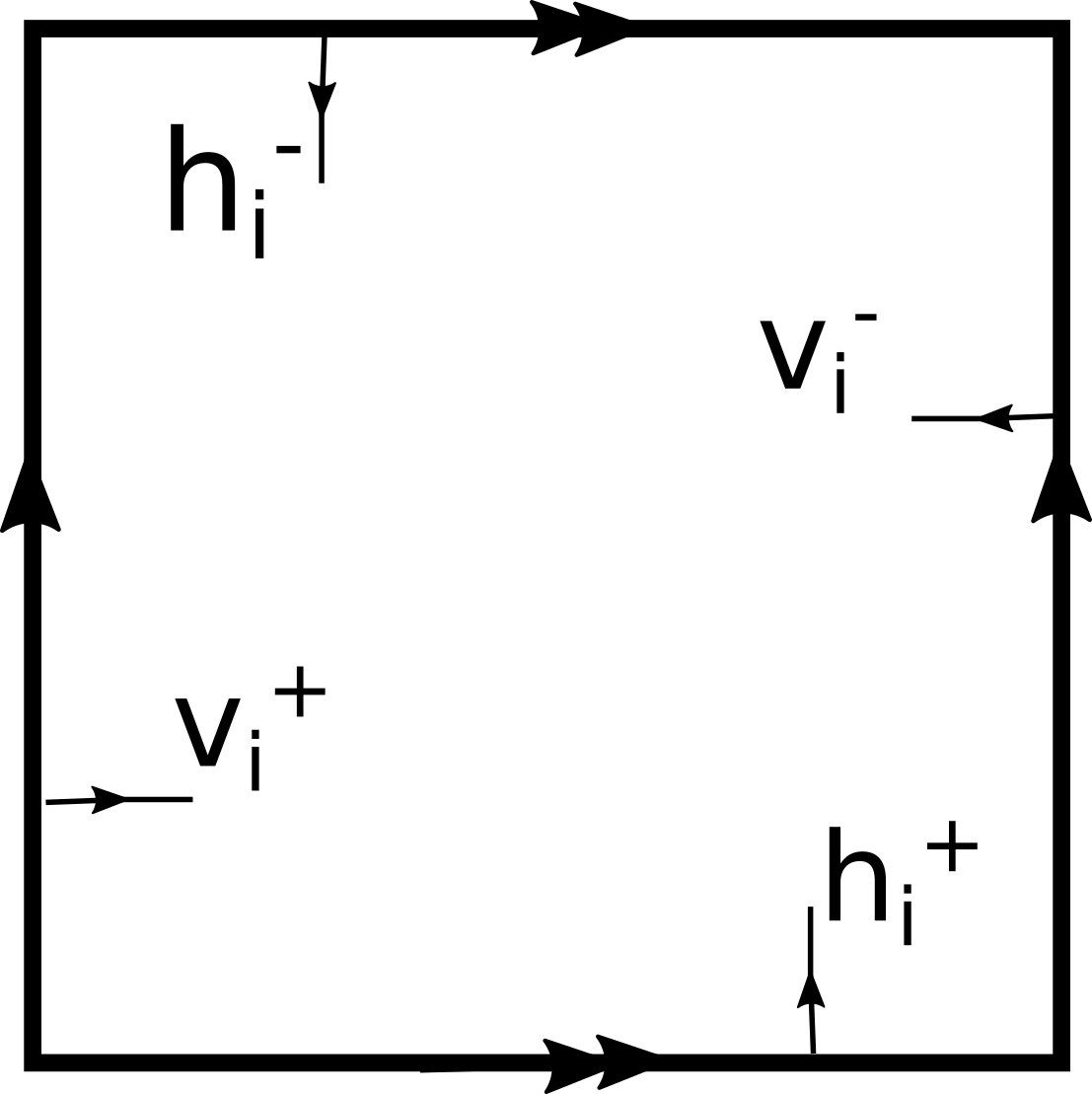}
\hspace*{0mm}
\includegraphics[height=27mm]{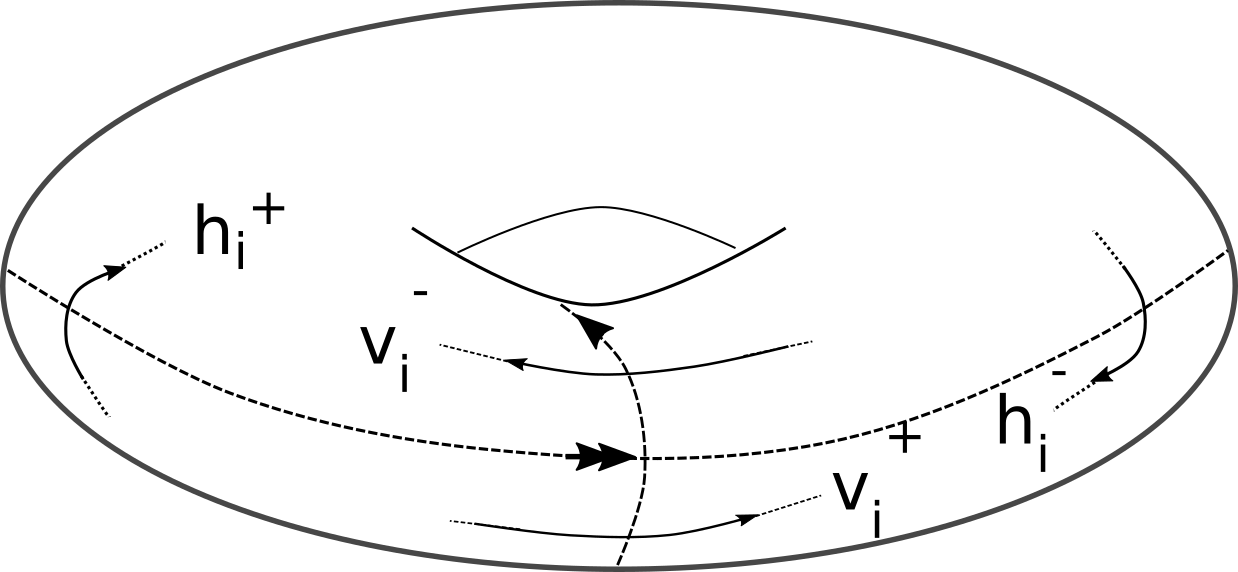}
\caption{Signs of horizontal and vertical crossings in a textile code, as in Definition~\ref{dfn:textile_code}, depicted in the torus (\emph{right}) and on a torus diagram (\emph{left})}
\label{fig:signs_hv}
\end{figure}

\begin{figure}[htb]
\begin{center}
\includegraphics[width=\textwidth]{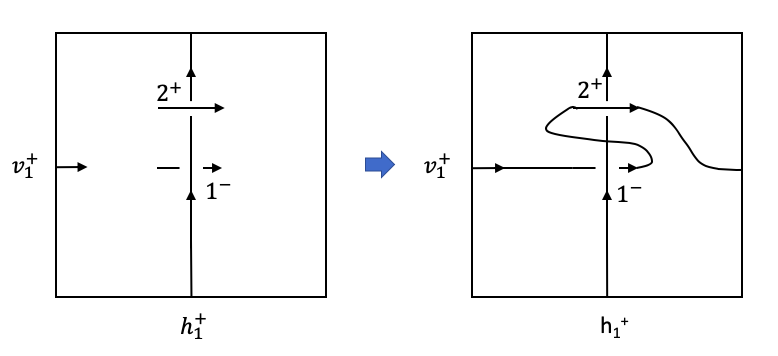}
\caption{This abstract textile code cannot be realized by a torus diagram.}  
\label{fig:unrealizable_codes}
\end{center}
\end{figure}

Though any word in a code is considered up to a cyclic permutation, we can start any word with the symbol that has a smallest integer, for example, from $h_1$ or $v_1$ (with any superscripts) or $1$ (for an overcrossing) or $1^{\pm}$ (for an undercrossing).
 
\begin{dfn}[abstract codes]
\label{dfn:abstract_code}
An \emph{abstract} textile code is a set of cyclic words made up of symbols $i, i^\pm, h_j^\pm, v_k^\pm$ with $i,j,k\in\N$ such that every $h_j,v_k$ has a unique superscript and every $i$ appears with exactly one of $\{i^+, i^-\}$.
\bs
\end{dfn}

An abstract textile code may not represent a torus diagram of a textile. 
Fig.~\ref{fig:unrealizable_codes} shows an attempt to realize the textile code $\{h_1^+ 1 2^+, v_1^+ 1^- 2\}$. 
Each symbol uniquely describes a local pattern around crossings or intersections with square boundaries.
However, these local patterns force a virtual crossing  (see Fig.~\ref{fig:unrealizable_codes}).

\begin{dfn}[complexity of a code and a link]
\label{dfn:complexity}
Let $T$ be an abstract textile code containing symbols $i, h_j^\pm, v_k^\pm$, with $i, j, k \in \N$. The \emph{complexity} of $T$ is defined as
\[
\max(i) + \max(j) + \max(k)
\]

The \emph{complexity} of a textile or a link $L$ in $T^2\times I$ is the minimum complexity of textile codes over all torus diagrams representing $L$. 
\bs
\end{dfn}

An oriented version of the textile in Fig.~\ref{fig:diagonal_textile} has the code as $\{h_1^+1v_2^-2^+, h_2^+v_1^+1^-2\}$ of complexity 4, see Fig.~\ref{fig:corner_turn} (right).  

\section{A topological realizability criterion for textile codes}
\label{sec:criterion}

To develop a realizability criterion for a given abstract textile code $W$, we introduce the textile graph $\Gamma(W)$ in Definition~\ref{dfn:graph_code}. This graph is then used as the skeleton of a $2$-dimensional CW-complex $\mathcal{S}(W)$. We will show that if the code does not give rise to virtual crossings, $\mathcal{S}(W)$ is a torus, and that it is possible to determine this algorithmically.

\begin{dfn}[textile graph]
\label{dfn:graph_code}
Let $W$ be an abstract textile code containing crossing symbols $1,\dots,n$ and intersection symbols $h_1,\dots,h_l$, $v_1,\dots,v_m$ (with superscripts), see Definition~\ref{dfn:abstract_code}.
The \emph{textile graph} $\Gamma(W)$ has $n+l+m$ vertices labelled with $1,\dots,n$, $h_1,\dots,h_l$, $v_1,\dots,v_m$ as in $W$, plus one \emph{corner} vertex labeled with $c$, see Fig.~\ref{fig:code_to_cycles}.
The \emph{extended} code $\bar W$ is obtained from $W$ by adding the \emph{boundary word} $c h_1\dots h_l c v_1 \dots v_m$.
For any adjacent pair of symbols in the extended code $\bar W$, add the unoriented edge between the corresponding vertices of the graph $\Gamma(W)$. All vertices representing crossings are in this way connected either to an adjacent crossing or to a vertex representing the corner of the torus diagram.
\bs
\end{dfn}

If $W$ is the textile code of a torus diagram $D$, then edges of $\Gamma(W)$ are continuous arcs between corresponding crossings or intersection points. Any components of $D$ would be represented by words in the code that do not intersect the boundary of the diagram. We may therefore avoid such components by insisting that each word contains some symbol $h_i^\pm$ or $v_i^\pm$.
The extra edges come from the boundary of the square, first going from left to right along the bottom edge, then from bottom to top along the right edge.
Recall that opposite edges of a square are identified to get a torus, see Fig.~\ref{fig:diagonal_textile}.
 
\begin{dfn}[textile complex]
\label{dfn:complex_code}
Let $\Gamma (W)$ be the textile graph of an abstract textile code $W$, see Definition~\ref{dfn:graph_code}.
For each unoriented edge $(a,b)$ between vertices $a,b\in\Gamma (W)$, we define two oppositely oriented edges labelled with $(a, b)_+$ and $(b, a)_-$.
The positive subscript in $(a, b)_+$ means that $b$ (cyclically) follows $a$ in $\bar W$.
The negative subscript in $(b, a)_-$ means that the symbol $b$ should (cyclically) precede $a$ in $\bar W$.
For each oriented edge, we define the next oriented edge by the symbolic rules: 
$$
(a, i)_\delta \to (i^{\epsilon}, b)_{\epsilon \delta}\;
(a, i^\epsilon)_\delta \to (i, b)_{-\epsilon\delta},\;$$
$$
(a, h_i)_\delta \to (h_i^\epsilon, b)_{\delta\epsilon},\;
(a, h_i^\epsilon)_\delta \to (h_i, b)_{-\delta\epsilon},\;$$
$$
(a, v_i)_\delta \to (v_{i}^\epsilon, b)_{-\delta\epsilon},\;
(a, v_i^\epsilon)_\de \to (v_i, b)_{\delta\epsilon},\;$$
$$
(h_i, c)_\epsilon \to (c, v_j)_\epsilon,\;
(v_i, c)_\epsilon \to (c, h_j)_{-\epsilon}, 
$$
where $\ep,\de\in\{-1,+1\}$ and $a,b$ denote any symbols of the textile code $W$ or, equivalently, vertices of $\Ga(W)$.
Fig.~\ref{fig:turn_at_crossing},\ref{fig:h_turn},\ref{fig:v_turn}, \ref{fig:corner_turn} explain the geometric meaning of choosing the next oriented edge by turning left in a torus diagram represented by the code $W$.
\smallskip

Since the symbolic rules above are written for any abstract code $W$, 
they can define  (in principle) closed oriented cycles so that every edge of $\Ga(W)$ is passed once in each direction.
The \emph{textile complex} $\mathcal{S}(W)$ is obtained by attaching a topological disk (along its boundary) to each resulting oriented cycle.
\bs 
\end{dfn}

\begin{figure}[htb]
\centering
\includegraphics[width=\textwidth]{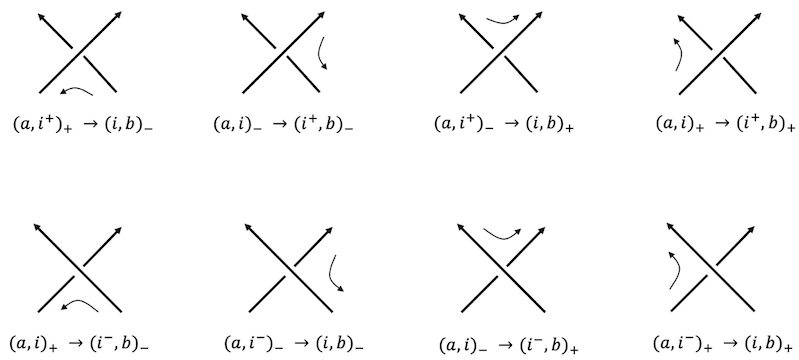}
\caption{Symbolically choosing a next edge by turning left at a crossing in Definition~\ref{dfn:complex_code}:
$a,b$ denote the vertices of $\Ga(W)$ before and after the shown crossing. }  
\label{fig:turn_at_crossing}
\end{figure}

\begin{figure}[htb]
\centering
\includegraphics[width=\textwidth]{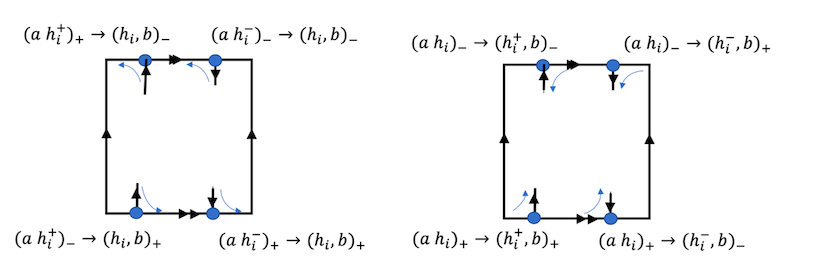}
\caption{Choosing a next edge by turning left at a horizontal intersection in Definition~\ref{dfn:complex_code}: $a,b$ denote the vertices of $\Ga(W)$ before and after the shown point.}
\label{fig:h_turn}
\end{figure} 

\begin{figure}[htb]
\centering
\includegraphics[width=\textwidth]{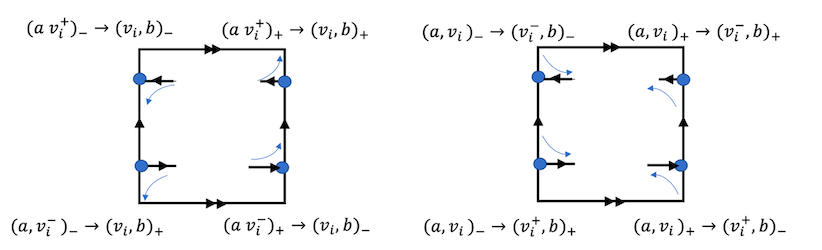}
\caption{Choosing a next edge by turning left at a vertical intersection in Definition~\ref{dfn:complex_code}: $a,b$ denote the vertices of $\Ga(W)$ before and after the shown point.}
\label{fig:v_turn}
\end{figure} 

\begin{figure}[htb]
\centering
\includegraphics[height=36mm]{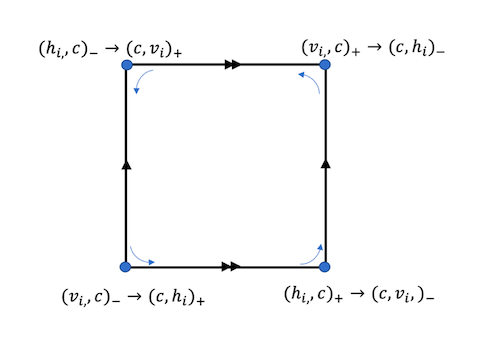}
\hspace*{0mm}
\includegraphics[height=36mm]{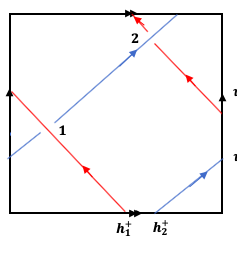}
\caption{\emph{Left}: turning left at a corner in Definition~\ref{dfn:complex_code}.
\emph{Right}: the (oriented) textile structure from Fig.~\ref{fig:diagonal_textile} has the textile code 
$W=\{h_1^+ 1 v_2^- 2^+, h_2^+ v_1^+ 1^- 2 \}$.}
\label{fig:corner_turn}
\end{figure} 

Lemma~\ref{lem:oriented_cycles} proves one part of the realizability in Theorem~\ref{thm:criterion}.

\begin{lem}
\label{lem:oriented_cycles}
Let $W$ be the textile code of a torus diagram $D\subset T^2$.
Let $\bar D$ be obtained from $D$ by adding the meridian and longitude of $T^2$ that map to the square boundary of $D$.
Then the textile complex $\mathcal{S}(W)$ is $T^2$ and $T^2-\bar D$ splits into open disks whose boundaries are oriented cycles from Definition~\ref{dfn:complex_code}.
\end{lem}

\begin{proof}
If the code $W$ describes the diagram of a textile $D\subset T^2$, then the complement $T^2-\bar D$ splits into disjoint curved polygons, see Fig.~\ref{fig:oriented_cycles}.
The extended diagram $\bar D$ includes a meridian and longitude to cover the case of improper textiles in Fig.~\ref{fig:improper_textiles}, when $T^2-D$ can contain pieces that are not topological disks.
By Definition~\ref{dfn:graph_code} the textile graph $\Ga(W)$ is the diagram $\bar D$, where all crossings become vertices and orientations of arcs are (temporarily) forgotten. Attaching disks along edges gives rise to a single connected component, since $\Ga(W)$, and hence $\mathcal{S}(W)$ are connected.
\smallskip

Following the `turn-left' rules in
Fig.~\ref{fig:turn_at_crossing}-\ref{fig:oriented_cycles}, we trace the (anticlockwisely oriented) boundary of every disk from $T^2-\bar D$.
By Definition~\ref{dfn:complex_code} the textile complex $\mathcal{S}(W)$ is homeomorphic to the torus $T^2$ obtained from $\Ga(W)\approx\bar D$ by attaching disks to the oriented cycles in $\Ga(W)$.
Since every arc of $\bar D^2$ belongs to the boundaries of two adjacent (consistently oriented) disks, every edge of the textile graph $\Ga(W)$ will be traced once in each of two opposite directions. 
\end{proof}

\begin{thm}[realizability of textile codes]
\label{thm:criterion}
An abstract textile code $W$ represents a realizable link $L$ in a thickened torus if and only if the textile complex $\mathcal{S}(W)$ is homeomorphic to a torus $T^2$. 
\end{thm}
\begin{proof}
The part `only if' $\Rightarrow$ is Lemma~\ref{lem:oriented_cycles}.
The part `if '$\Leftarrow$ assumes that $\mathcal{S}(W)$ is a torus $T^2$. 
We embed $\Ga(W)\subset\mathcal{S}(W)=T^2$ into $\R^3$ and remove the arcs between all vertices $h_i,v_j,c$.
To replace the remaining vertices by crossings, use their signs in $W$ and orientations from left-to-right orders on words of $W$.
\end{proof}

\begin{figure}[htb]
\centering
\includegraphics[height=42mm]{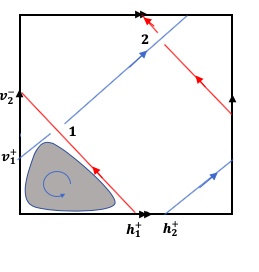}
\hspace*{0mm}
\includegraphics[height=42mm]{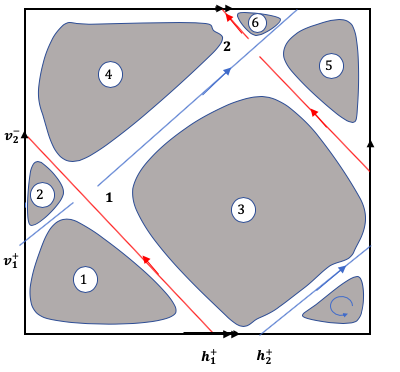}
\caption{\emph{Left}: the marked cycle has the edges $(h_1^+ 1)_+\to(1^- v_1^+)_-\to(v_1^+ c)\to(c h_1^+)$ following the `turn-left' rules in Definition~\ref{dfn:complex_code}.
\emph{Right}: All 7 cycles.}
\label{fig:oriented_cycles}
\end{figure} 

\section{A fast algorithm to test realizability of textile codes}
\label{sec:algorithm}

The realizability criterion in Theorem~\ref{thm:criterion} is algorithmically verified by the five steps below. 
Each step is illustrated by explicit calcuations on the textile of Fig.~\ref{fig:diagonal_textile}.
\smallskip

The algorithm as described will function for codes with any number of individual words, and hence textiles with any number of components. In this paper we employ it only on single word codes.

\noindent
\emph{Step 1}.
For any abstract textile code $W$, each pair of symbols $\{i,i^{\pm}\}$ generates a vertex (a future crossing).
Every symbol $h_i^{\pm},v_j^{\pm}$ generates its own vertex (a future intersection with a meridian and longitude).
We add a corner vertex $c$.
\smallskip

The textile code $W=\{h_1^+ 1 v_2^- 2^+, h_2^+ v_1^+ 1^- 2\}$ from Fig.~\ref{fig:oriented_cycles} generates two crossing vertices labelled with $1,2$ (without signs for simplicity); two (horizontal) vertices $h_1,h_2$; two (vertical) vertices $v_1,v_2$; one corner vertex $c$, see Fig.~\ref{fig:code_to_cycles}.  
\medskip

\noindent
\emph{Step 2}.
Generate unoriented edges of the textile graph $\Gamma(W)$ from pairs of successive symbols in $W$ and the extra word $ch_1\dots h_l v_1\dots v_m$ going from the corner vertex through first horizontal, then vertical vertices, see Definition~\ref{dfn:graph_code}.
\smallskip

The above code $W$ has the four edges $(h_1, 1), (1, v_2), (v_2, 2), (2, h_1)$ from the first word of $W$, four edges $(h_2, v_1), (v_1, 1), (1, 2), (2, h_2)$ from the second word of $W$ and six edges $(c, h_1)$, $(h_1, h_2)$, $(h_2, c)$, $(c, v_1)$, $(v_1, v_2)$, $(v_2, c)$ from the extra word, see Fig.~\ref{fig:code_to_cycles}.
\medskip

\begin{figure*}[htb]
\centering
\includegraphics[width=\textwidth]{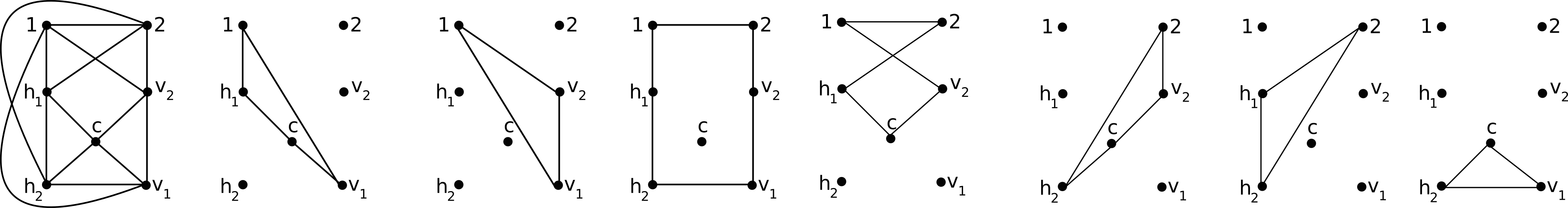}
\caption{The abstract textile code $W=\{h_1^+ 1 v_2^- 2^+, h_2^+ v_1^+ 1^- 2\}$ has the extra word $ch_1h_2cv_2v_2$.
 Its textile graph $\Ga(W)$ and all 7 cycles from Table~\ref{tab:oriented_cycles} are shown.}
\label{fig:code_to_cycles}
\end{figure*} 

\noindent
\emph{Step 3}.
Find oriented cycles in $\Ga(W)$ by the `turn-left' rules in Definition~\ref{dfn:complex_code} illustrated in Fig.~\ref{fig:turn_at_crossing},\ref{fig:h_turn},\ref{fig:v_turn}, \ref{fig:corner_turn}.
\smallskip

For the code $W$, we start from the pair of successive symbols $(h_1^+, 1)_+$, say with the positive subscript meaning that $1$ follows $h_1^+$ in the extended code $\bar W$. 
The next edge starts with $1^-$ and the `turn-left' rule $(a, i)_\delta \to (i^{\epsilon}, b)_{\epsilon \delta}$ for $a=h_1^+$, $i=1$, $\de=+1$, $\ep=-1$ implies that $b$ is adjacent to $1^-$.
Since $\de\ep=-1$, the symbol $b$ precedes $1^-$ in $\bar W$.
So $b=v_1^+$, the next edge is $(1^-, v_1^+)_-$.
\smallskip

The final vertex is $v_1^+$, so we will use the rule from Fig.~\ref{fig:v_turn}: $(a, v_i^\epsilon)_\de \to (v_i, b)_{\delta\epsilon}$ for $a=1^-$, $i=1$, $\ep=+1$, $\de=-1$.
Since $\de\ep=-1$, the next symbol $b$ should precede $v_1$ in the extended code $W$ including the extra word $ch_1h_2cv_1v_2$.
Then $b=c$, so the next edge is $(v_1, c)_-$, where the superscript of $v_1$ is skipped for a move along a vertical edge. 
The next edge is found by the corner rule $(v_i, c)_{\de}\to (c, h_j)_{-\de}$ with $i=j=1$, $\de=-1$ in Fig.~\ref{fig:corner_turn}.
The rule $(a, h_i)_{\de}\to (h_i^{\ep} b)_{\de\ep}$ for $i=1$, $\ep=+1$, $\de=-1$ gives $b=1$ following $h_1^+$ in $\bar W$. 
The first cycle in Table~\ref{tab:oriented_cycles} is complete. 
\smallskip

Selecting the unused edge $(1, v_2^-)_+$, we continue tracing edges and forming cycles until we get the other 6 cycles in Fig.~\ref{fig:oriented_cycles}.
\medskip

\begin{table*}[!t]
\caption{All 7 cycles obtained from $W=\{h_1^+ 1 v_2^- 2^+, h_2^+ v_1^+ 1^- 2\}$, see Fig.~\ref{fig:oriented_cycles} and~\ref{fig:code_to_cycles}.}
\label{tab:oriented_cycles}
\centering
\begin{tabular}{ |c|c| }
\hline
cycle & oriented edges in the cyclic order\\
\hline
$1$ & $(h_1^+, 1)_+\to (1^-, v_1^+)_-\to (v_1, c)_-\to (c, h_1)_+ $ \\
\hline
$2$ & $(1, v_2^-)_+\to (v_2, v_1)_-\to (v_1^+, 1^-)_+$ \\
\hline
$3$ & $(v_2^-, 2^+)_+\to (2, 1^-)_-\to (1, h_1^+)_-\to (h_1^+, h_2^+)_+\to (h_2^+, v_1^+)_+\to (v_1^+, v_2^-)_+$ \\
\hline
$4$ & $(2^+, h_1^+)_+\to (h_1, c)_+\to (c, v_2)_-\to (v_2^- ,1)_-\to (1^+, 2)_+$ \\
\hline
$5$ & $(2^+, v_2^-)_-\to (v_2, c)_+\to (c, h_2)_+\to (h_2^+, 2)_-$ \\
\hline
$6$ & $(h_1^+, 2)_-\to (2, h_2^+)_+\to (h_2, h_1)_+$ \\
\hline
$7$ & $(v_1^+, h_2^+)_-\to (h_2, c)_+\to (c, v_1)_+ $ \\
\hline
\end{tabular}
\end{table*}

\noindent
\emph{Step 4}.
Check that each oriented edge is passed once in each of two opposite directions.
If any contradictions emerge, then the code $W$ is unrealizable.
\medskip

\noindent
\emph{Step 5}.
Find the Euler characteristic of $\mathcal{S}(W)$ as $\chi=\#{\mbox{vertices of }\Ga(W)}-\#{\mbox{edges of }\Ga(W)}+\#{\mbox{oriented cycles of }\Ga(W)}$.
The torus $T^2$ is uniquely characterized up to a homeomorphism as a compact orientable surface with $\chi=0$ and without boundary.
Orientability and empty boundary were checked in Step 4.  
If $\chi=0$, the code is realizable, otherwise not.
\medskip

The pseudocode for the non-trivial steps of the algorithm is given in \ref{alg:periodB} whose linear complexity is proved below.

\begin{thm}[algorithm complexity]
\label{thm:algorithm}
Given an abstract textile code $W$ of a length $N$, the realizability algorithm has the complexity $O(N)$ to decide if $W$ represents a link in $T^2\times I$. 
\end{thm}
\begin{proof}
We read along the $N$ symbols of the code $W$ in both directions to generate the set of oriented edges giving $2N$ steps in total to generate $4N$ oriented edges. 
Each edge will be assigned to a cycle in a single step for a total of $4N$ steps.
The overall algorithm is therefore of linear complexity in $N$.\
\end{proof}

Note that for input codes, we replace the formal superscript and subscript symbols $\pm$ with the  values $\pm 1$ to enable arithmetic determination of adjacent edges in a cycle. 
\smallskip

Steps $1$ may be described trivially -  creation of the graph simply involves the listing of all distinct symbols. Edges are listed as distinct cyclically adjacent pairs in words.
\smallskip

Step $2$ is also trivial - for each edge $[a,b]$ in the graph, where $a$ an $b$ are distinct symbols (possibly with subscripts and superscripts) in the code, two oriented arc symbols $[a, b]_+$ and $[b, a]_-$ are generated
\smallskip


\begin{table*}[!t]
\caption{Enumeration of all abstract and realizable textile codes of up to complexity $5$ as introduced in Definition~\ref{dfn:complexity} up to cyclic permutations.}
\centering
\begin{tabular}{ |c|c|c|c|c|c| }
\hline
Complexity &  Crossings & Horizontal Points & Vertical Points &Abstract Textile Codes & Realizable Textile Codes\\
\hline
$3$ & $1$ & $1$ & $1$ & $48$ & $8$ \\
\hline
$4$ & $2$ & $1$ & $1$ & $1920$ & $672$\\
\hline
$5$ & $1$ & $3$ & $1$ & $3840$ & $368$\\ 
\hline
$5$ & $2$ & $2$ & $1$ & $23040$ & $2816$ \\
\hline
$5$ & $3$ & $1$ & $1$ & $161280$ &  $24960$\\
\hline
\end{tabular}
\label{tab:allcodes}
\end{table*}
\smallskip

Algorithm~\ref{alg:periodB} covers steps $3$ to $5$. Contradictions -- the appearance of opposite oriented edges in the same cycle or edges that cannot be assigned -- return a false result, otherwise the algorithm counts complete cycles and tests that the Euler characteristic is $0$. 
There are twice as many unoriented edges as vertices, so $\chi=0$ is equivalent to checking that the number of cycles is equal to the number of distinct symbols in the code. 
\smallskip

The algorithm coded in C++ and can be made available in May 2020.
\smallskip
Sections~\ref{sec:reductions} and \ref{sec:classification} will use these algorithms to classify all links in a thickened torus up to complexity 5.

\begin{algorithm}
\KwIn{EDGEPAIR - set of oriented edges}
\KwOut{True/False flag for realizability}
\medskip
\tcp{Start cycle count}
$cycle \leftarrow 0$\;
\tcp{Initiate the first cycle}
CYC $\leftarrow \emptyset$\;
inputpair  $\leftarrow (a,b)_\delta, \delta = \pm 1 \in EDGEPAIR$\;
\tcp{Runs until all edges have been used}
\While {EDGEPAIR $\neq$ USEDPAIR}{
        \tcp{runs while the next input pair is not already part of the cycle. Returns FALSE if a cycle contains both oriented passes of a graph edge.}
		\While{inputpair $\notin$ CYC}{
			\tcp{Selects edge next to \textit{inputpair}}
			\uCase{$b=i, i \in \N$}{
				Set {outputpair} $\leftarrow (i^\epsilon, d)_{\delta\epsilon}$ from EDGEPAIR\;
				\lIf {$(d, i^\epsilon)_{-\delta\epsilon} \in$ CYC}{\Return False 	
				}
			\uCase{$b = i^\epsilon, i \in \N, \epsilon = \pm 1$}{
				\textit{outputpair} $\leftarrow (i, d)_{-\delta\epsilon}$ from EDGEPAIR\;
				\lIf {$(d, i)_{\delta\epsilon} \in$ CYC}{\Return False }	
			}
			\uCase{$b = h_j,  j \in \N$}{
				\textit{outputpair} $\leftarrow (h_j^\epsilon, d)_{-\delta\epsilon}, \epsilon = \pm 1$\;
				\lIf {$(d, h_j^\epsilon)_{\delta\epsilon} \in $ CYC}{\Return False}					
				}
			\uCase{$b = h_j^\epsilon, j \in \N, \epsilon= \pm 1$}{
				\textit{outputpair} $ \leftarrow (h_j, d)_{-\delta\epsilon}$
			}
			\uCase{$b = v_k, k \in \N$}{
				\textit{outputpair} $\leftarrow (v_k^\epsilon, d)_{-\delta\epsilon}, \epsilon = \pm 1$\;	
				\lIf {$(d, v_k^\epsilon)_{\delta\epsilon} \in$ CYC}{\Return False}		
				}
			\uCase {$b = v_k^\epsilon, k \in \N, \epsilon = \pm 1$}{
				\textit{outputpair} $\leftarrow (v_k, d)_{\delta\epsilon}$
			}
			\uCase {$b = c$, $a = h$}{
			    \textit{outputpair}$ \leftarrow (c, v_i)_{\delta'}$
			   }
		    \uCase {$b = c$, $a = v$}{
			    \textit{outputpair}$ \leftarrow (c, h_i)_{\delta'}$
			    }
			\tcp{adds next pair to current cycle} 
			CYC $\leftarrow \textit{outputpair}$\;
			\tcp{continues with selected pair as input pair}
			\textit{inputpair} $\leftarrow$ \textit{outputpair}\;
			}
		\tcp{Increments count when cycle complete}
		\uIf  {$|\text{CYC}| > 1$}
			{$cycle \leftarrow cycle +1$
			\tcp{tracks which edges have been used in a cycle}
			 USEDPAIR $\leftarrow$ CYC\;
			  CYC $\leftarrow \emptyset$}
			  \tcp{starts new cycle with unused oriented edge}
			  \textit{inputpair} $\leftarrow (a,b)_\delta, \delta = \pm 1 \in$ EDGEPAIR and $\notin$ USEDPAIR
			\Else{\Return False}
			}
		}
		
\tcp{Checks Euler characteristic}
	\Return cycle $= \max(i) + \max(j) + \max(k) + 1$
\caption{Steps $3$ to $5$ of the algorithm in section~\ref{sec:algorithm}: builds oriented cycles in $\Ga(W)$ and tests that the textile complex $\mathcal{S}(W)$ from Definition~\ref{dfn:complex_code} is homeomorphic to $T^2$.}
\label{alg:periodB}
\end{algorithm}

\section{Reductions of textile codes by Reidemeister moves}
\label{sec:reductions}

The realizability algorithm in section~\ref{sec:algorithm} substantially reduces the number of abstract codes of a given complexity.  
We are still left with a very large number of codes - certainly too large to investigate their isotopic classes. 
However, we have implemented further reductions by using Reidemeister moves in Fig.~\ref{fig:Reidemeister}.
\smallskip


For the diagram of a realizable textile code, the numbering of crossing vertices is arbitrary. 
For a diagram of $n$ crossings, any action of the permutation group $S_n$ on the crossings will result in an identical diagram, which justifies the reductions below. 
\smallskip

We will group our realizable codes into those that are just permutations of each other and eliminate all of these except - by convention - the code where crossings are numbered in ascending order in terms of where they are first encountered.
\smallskip

We also note that there are some cases where the complexity of a code from Definition~\ref{dfn:complexity} is greater than that of the link it describes - that is, the number of crossings in a textile may be reduced via Reidemeister moves. 
Fig.~\ref{fig:reidcodes} shows the patterns in the code reducible by Reidemeister moves I and II. 

\begin{figure}[htb]
\centering
\includegraphics[width=\textwidth]{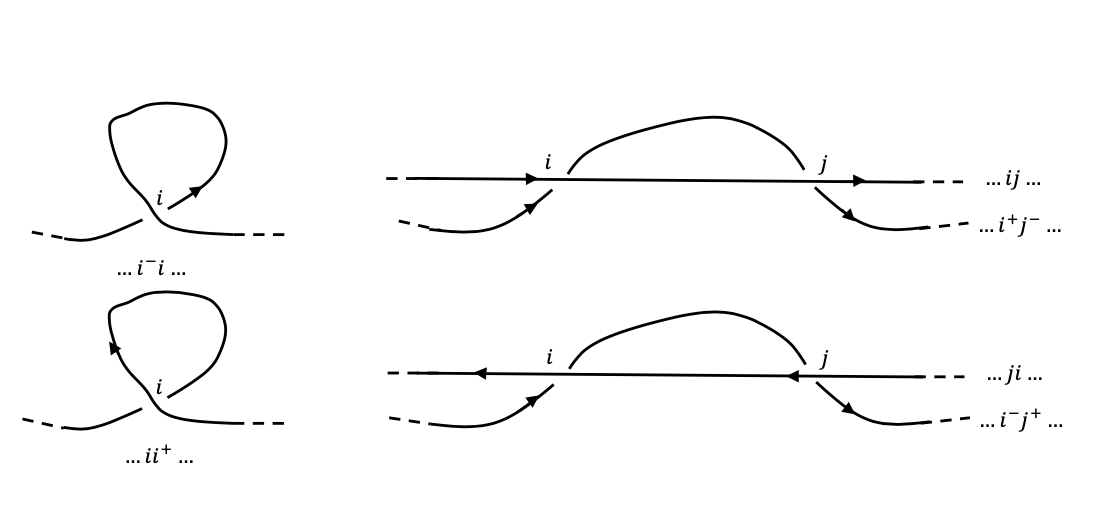}
\caption{Codes indicating applicability of Reidemeister moves in Fig.~\ref{fig:Reidemeister}.  
\emph{Left} : the application of Reidemeister move I.
\emph{Right}: the application of move II.} 
\label{fig:reidcodes}
\end{figure}

Any diagrams whose codes contain a pattern $\ldots i i^\pm \ldots$ or $\ldots i^\pm i \ldots $ represent links isotopic to links of a lower complexity by Reidemeister move I in the left hand side picture of Fig.~\ref{fig:reidcodes}.
Similarly, any realizable codes which contain a pattern  $ \ldots ij \ldots i^+ j^- \ldots$ or 
$ \ldots ji \ldots i^- j^+ \ldots$ represent knots isotopic to those of lower complexity by Reidemeister move II in Fig.~\ref{fig:reidcodes}. 
\smallskip

The implementation of these steps substantially reduced the number of codes, such that we were able to investigate the configuration of proper textiles they represented directly. 
\smallskip

Note that all single component textiles of complexity $3$ were eliminated at this stage - their only crossing is a self-crossing and therefore resolvable by Reidemeister move I. 
The remaining codes are given in the column labelled 'reduced textile codes' in Table~\ref{tab:redcodes}, where the numbers  are much smaller than in the final column of Table~\ref{tab:allcodes}.
\smallskip

\begin{table*}[!t]
\caption{Enumeration of all textile codes of complexity $\leq 5$ after reduction by vertex permutation and Reidemeister moves}
\centering
\begin{tabular}{ |c|c|c|c|c|c| }
\hline
Complexity &  Crossings & Horizontal Points & Vertical Points &Reduced textile codes & Isotopically distinct classes\\
\hline
$4$ & $2$ & $1$ & $1$ & $8$ & $8$\\
\hline
$5$ & $2$ & $2$ & $1$ &  $48$ & $16$\\ 
\hline
$5$ & $3$ & $1$ & $1$ & $32$ & $8$\\
\hline
\end{tabular}
\label{tab:redcodes}
\end{table*}


The much smaller numbers of textiles can be investigated geometrically to determine whether any were equivalent under ambient isotopies, or had any other relationships. 
\smallskip

For codes of complexity $4$ with $2$ crossings, none of the eight realizable codes were found to be isotopic to each other. 
\smallskip

Codes of complexity $5$ with $2$ crossings were found to be related to textiles given by codes of complexity $4$ via a homeomorphism of the torus called a \emph{Dehn twist}. They are thus distinct as embeddings in a fixed torus.
\smallskip

Fig.~\ref{fig:dehntwist} demonstrates such a twist - it is also possible to have a twist in the opposite orientation. 
Note that the addition of such a twist is visible in the textile code as an additional $v_i^\pm$ symbol.
\smallskip

A Dehn twist is not, however, an ambient isotopy of the embedded link. 
Each such textile, however, was found to be ambient isotopic to two others via a simple translation (involving no Reidemeister moves), see Fig.~\ref{fig:dehniso}. 
This reduces the number of non-isotopic classes from $48$ to $16$

\begin{figure}[htb]
\centering
\includegraphics[width = \textwidth]{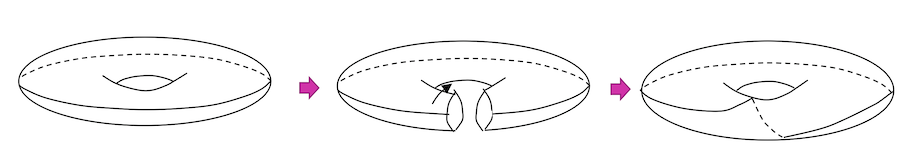}
\hspace*{0mm}
\includegraphics[width = \textwidth]{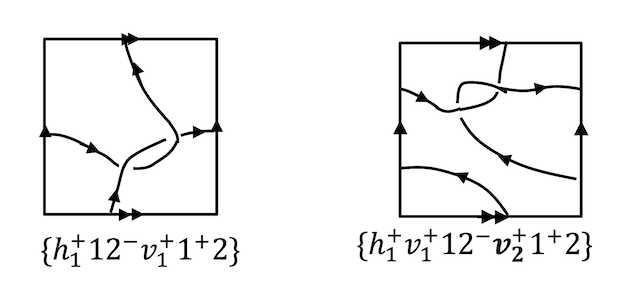}
\caption{\emph{Top}: A meridional Dehn twist
\emph{Bottom}: A complexity $5$ textile code with two crossings related to a textile with fewer crossings via a Dehn twist.}
\label{fig:dehntwist}
\end{figure} 

Sixteen of the codes of complexity $5$ with $3$ crossings were found to be related to a $2$ crossing textile by a pair of Reidemeister moves shown in Fig.~\ref{fig:3to2} 
\smallskip

\begin{figure}[htb]
\begin{center}
\includegraphics[width = \textwidth]{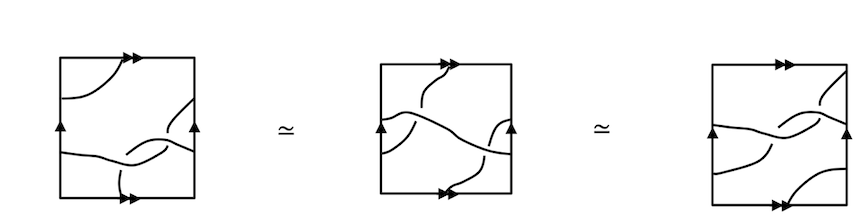}

\includegraphics[width = \textwidth]{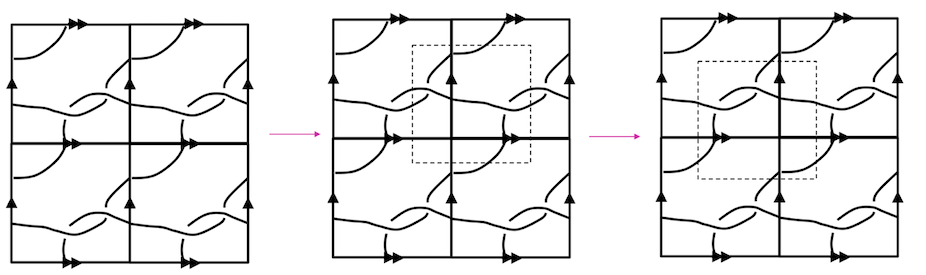}
\caption{\emph{Top}: A meridional Dehn twist
\emph{Bottom}: Three complexity $5$ textile stuctures with two crossings (in a square) are ambient isotopic to each other.}
\label{fig:dehniso}
\end{center}
\end{figure} 

The remaining sixteen were found to represent eight distinct textiles, with pairs of textiles related by an isotopy which 'rotates' the embedded link as shown in Fig.~\ref{fig:3torinv}.

\begin{figure}[htb]
\centering
\includegraphics[width=\textwidth]{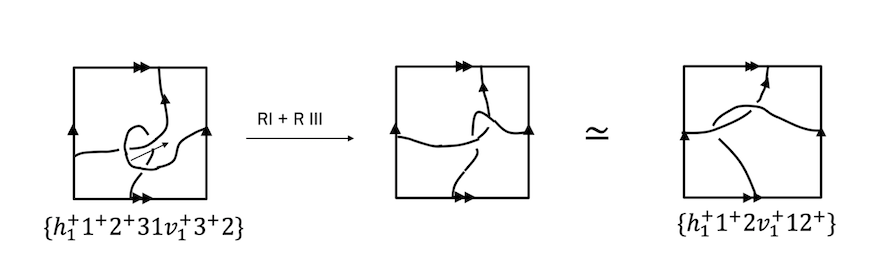}
\caption{A textile of complexity $5$ with $3$ crossings which is ambient isotopic to one of complexity $4$ with $3$ crossings.} 
\label{fig:3to2}
\end{figure}

\begin{figure}[htb]
\centering
\includegraphics[width=\textwidth]{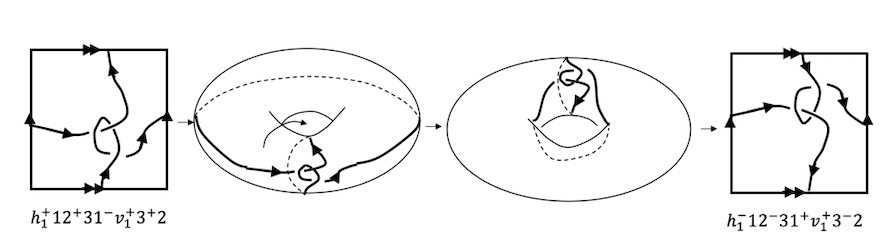}
\caption{Two knots of complexity $5$ that are ambient isotopic} 
\label{fig:3torinv}
\end{figure}

\section{An isotopic classification of textiles up to complexity $5$.}
\label{sec:classification}

To label the remaining textiles, we use a short notation $n^k_{(x,y)}$, where $n$ is the number of crossings, $k$ the number of components and $(x,y)$ represents the homology class expressed as the total sum of meridional and longitudinal windings of a link component (signed for their orientations) around the torus.
\smallskip

The symbol $\tilde{n}$ is used to denote a knot related to another by a swap of  under- and overcrossings.
The homology class subscript will normally distinguish textiles related to each other by orientation reversal. However, in the special case where two knots whose diagrams present as mirror images of each other have homology class $(a, 0)$ or $(0,b)$ we will differentiate these by placing a bar over the $0$.
\smallskip

Fig.~\ref{fig:symbols} illustrates these considerations in the case of two distinct textiles of complexity $5$ with two crossings.

\begin{figure}[htb]
\centering
\includegraphics[width = \textwidth]{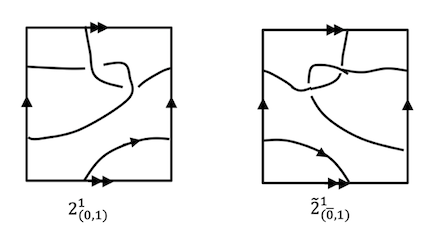}
\caption{Symbols denoting two textiles represented by codes of complexity $5$.} 
\label{fig:symbols}
\end{figure}


To test that we have indeed all non-isotopic knots, we calculate the \emph{Zenkina polynomial} associated with each knot \cite{zenkina2016invariant}. 
The Zenkina polynomial can be defined for knot diagrams in any orientable surface. 
Definition~\ref{dfn:zenkpoly} will be specific to single component knots in a thickened torus. 
\smallskip

The Gauss code of a link gives rise to a \emph{Gauss diagram}, in which the symbols $i$ are placed at points around a circle in the order in which they appear in the code, and chords are drawn across the circle joining identical symbols, see Fig.~\ref{fig:trefgauss}.
To get the Gauss diagram of a textile code, we similarly place crossing labels around the circle while ignoring any $h_i^\pm$ and $v_i^\pm$ symbols. 

\begin{figure}[htb]
\begin{center}
\includegraphics[height=25mm]{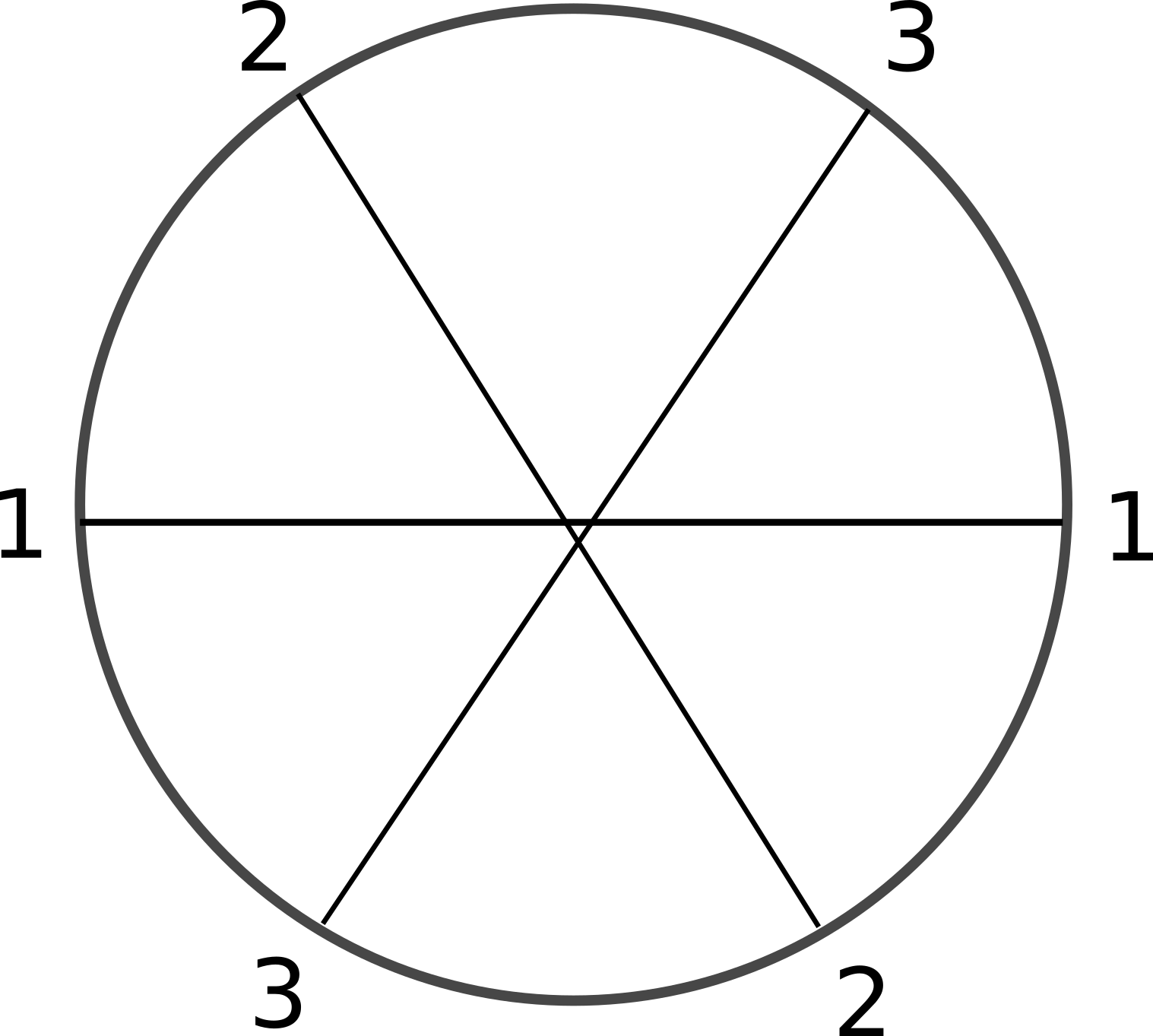}
\hspace*{3mm}
\includegraphics[height=25mm]{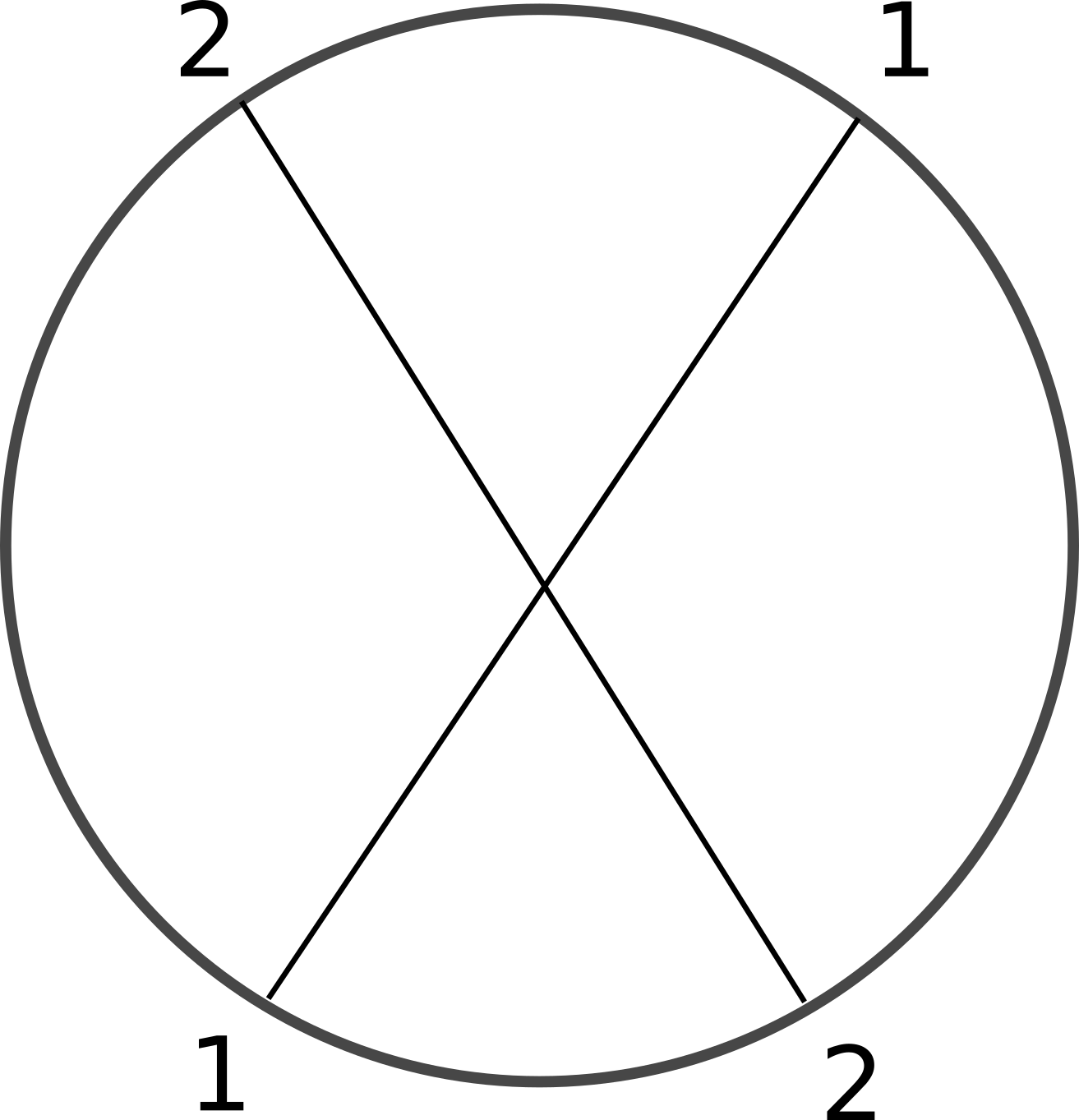}
\caption{Gauss diagrams of the classical (left) and virtual (right) knots in Fig.~\ref{fig:Gauss_codes}.
All crossings are pairwisely linked in both Gauss diagrams by Definition~\ref{dfn:arcdegree}. 
}
\label{fig:trefgauss}
\end{center}
\end{figure}

\begin{dfn}[parity of crossings, degrees $(\alpha,\beta)$ of sub-arcs]
\label{dfn:arcdegree}
Let $D$ be a torus diagram of a knot in a thickened torus $T^2\times I$.
\smallskip

Considering the Gauss diagram $G$ arising from the Gauss code $W$ of $D$, crossings $i,j$ are called \emph{linked} if their corresponding chords $i,j$ intersect in the Gauss diagram, see Fig.~\ref{fig:trefgauss}.
\smallskip

The \emph{parity} function $f$ \cite{Manturovparity} is defined for a crossing $i$ as $f(i) = 0$ if the crossing is linked with an even number of other crossings and $f(i) = 1$ if $i$ is linked with an odd number of crossings.
\smallskip

The $i$-th \emph{arc} $a_i$ of $D$ is a curve beginning at an undercrossing labelled $i$ and ending at some other undercrossing $j$.
Intersections of an arc $a_i$ with the edges of $D$ may divide $a_i$ into several \emph{sub-arcs} ordered according to the orientation of the arc $a_i$.
\smallskip

The \emph{homology degree} $(\alpha, \beta) \in \Z^2$ of the first sub-arc of $a_i$ is defined as $(0,0)$. 
The values $\alpha$ and $\beta$ for every next sub-arc of $a_i$ is increased by $1$ from the previous sub-arc when crossing (respectively) the vertical or horizontal edges in the positive direction as defined in Fig.~\ref{fig:signs_hv}, and decreased by $1$ when crossing these edges in the negative direction.
We label each sub-arc of $a_i$ with its homology degree as a suffix: $a_i(\alpha, \beta)$
\bs
\end{dfn}

As an example we consider the single component textile with code $\{h_1^+12^+31^-v_1^+3^+2\}$ in Fig.~\ref{fig:zenkexample}. 
We show this textile with crossings and sub-arcs labelled by homology degrees.
\smallskip

An arc can be read directly  as oriented fragments of the code running cyclically from its numbered undercrossing to the next undercrossing.
The arc $a_1$ given by the substring $1^-v_1^+3^+$ starts from undercrossing $1^-$ and consists of sub-arc $a_1(0,0)$ followed by sub-arc $a_1(1,0)$ finishing at undercrossing $3^+$.
The arc $a_2$ is given by the substring $3^+2h_1^+12^+$. 
The presence of the symbol $h_1^+$ divides $a_2$ into two sub-arcs given by the substrings $3^+2h_1^+$ and $h_1^+12^+$ with homology degrees $(0,0)$, $(0,1)$ respectively. 
\smallskip

Fig.~\ref{fig:zenkexample} also shows the Gauss diagram of this textile - note that crossing $1$ is even, while the other two crossings are odd. 

\begin{figure}[htb]
\begin{center}
\includegraphics[height=39mm]{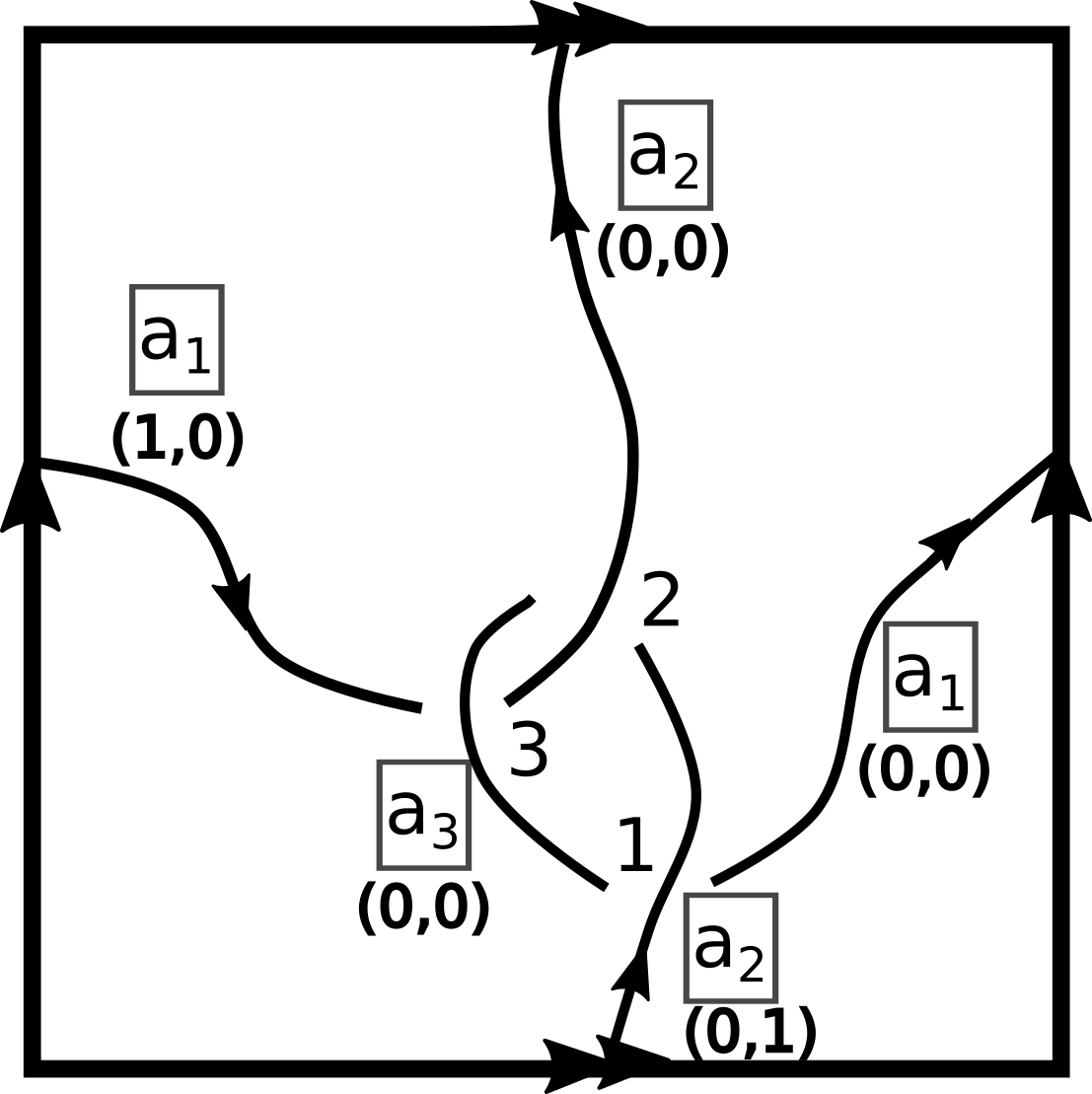}
\hspace*{2mm}
\includegraphics[height=39mm]{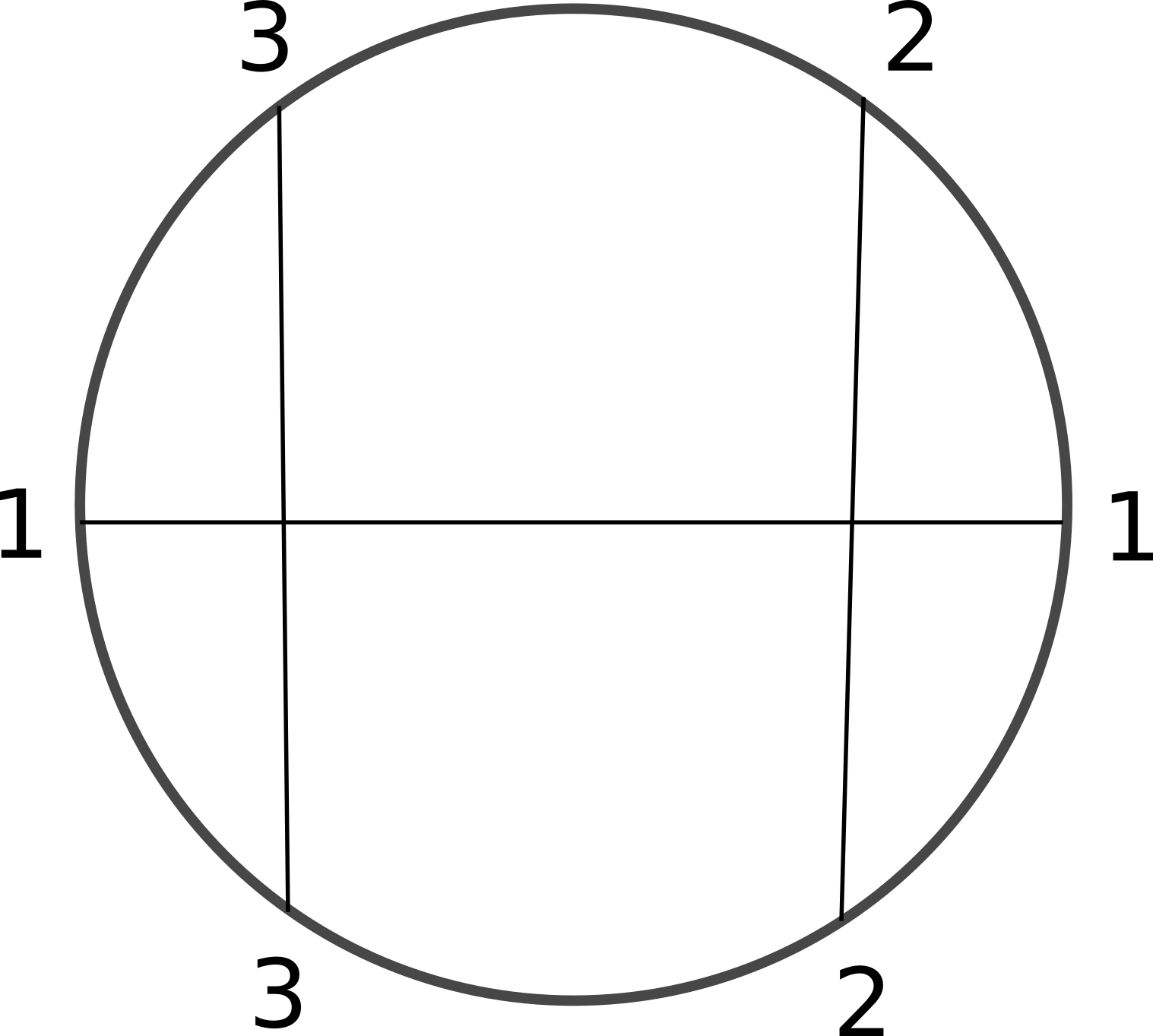}
\caption{\emph{Left}: The knot with code $\{h_1^+12^+31^-v_1^+3^+2\}$. All sub-arcs are labelled with homology degrees $(\alpha,\beta)$ from Definition~\ref{dfn:arcdegree}.
\emph{Right} : the Gauss diagram showing crossings $1,2,3$ by chords along a circle parameterizing the knot.}
\label{fig:zenkexample}
\end{center}
\end{figure}

\begin{dfn}[Zenkina polynomial]
\label{dfn:zenkpoly}
For a crossing $i$ and an arc $a_j$, possibly divided into sub-arcs, the \emph{incidence factor} is
\[
[i:a_j] :=  \epsilon_1z_1x^{\alpha_1} y^{\beta_1} + \epsilon_2z_2x^{\alpha_2} y^{\beta_2} + \epsilon_3z_3x^{\alpha_3} y^{\beta_3}
\]
where
$\epsilon_1$ is equal to $1$ if the arc $a_j$ leaves the crossing $i$ and $0$ otherwise; $\epsilon_2$ is equal to $1$ if the arc $a_j$ goes through the crossing (which should be an overcrossing) and $0$ otherwise; $\epsilon_3$ is equal to $1$ if the arc $a_j$ terminates at the crossing $i$ and $0$ otherwise. 
\smallskip

The values of exponents $\alpha_k, \beta_k$ above are given by the homology degree of the sub-arc of arc $j$ which leaves, passes through or arrives at the crossing $i$, respectively for $k=1,2,3$.
\smallskip

The values of the coefficient $z_2$ are given by (see Fig.~\ref{fig:zenkparity})
\[
z_2 = 
\begin{cases}
1-t & \text{if the parity } f(i) = 0\\
q & \text{if the parity } f(i) = 1.
\end{cases}
\]

If the sign of the crossing $i$ is positive, then 
\begin{align*}
z_1  & = -1\\
z_3 & = 
\begin{cases}
t & \text{if the parity } f(i) = 0\\
p & \text{if the parity } f(i) = 1.
\end{cases}
\end{align*}

If the sign of $i$ is negative, the values of $z_1,z_3$ are exchanged. 
The matrix $A(D)$ of a diagram $D$ consists of the entries
\[
A_{ij} = [i:a_j]
\]
The \emph{Zenkina polynomial} is the determinant $|A(D)|$ considered in the ring
$
\Z[p,q,t,x,y]/\langle q^2 - (1-t)(1-p), qp-qt\rangle
$
\bs
\end{dfn}

The two polynomial relationships above ensure the invariance under the third Reidemeister move.
The Zenkina polynomial is proved in~\cite{zenkina2016invariant} to be invariant under ambient isotopy in $T^2\times I$ up to any monomial factor $\pm p^\alpha q^\beta t^\gamma$ with $\alpha, \beta, \gamma \in \Z$.
\smallskip

\begin{figure}[H]
\centering
\includegraphics[height = 40mm]{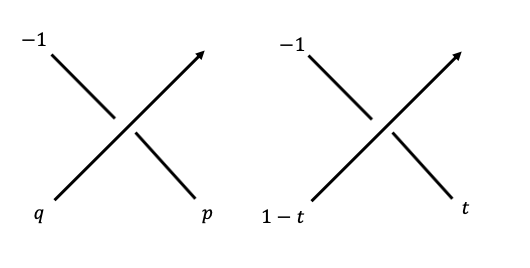}
\caption{Values of $z_i$ in an odd (left) and even (right) crossing} 
\label{fig:zenkparity}
\end{figure}

Continuing our calculation of the Zenkina polynomial for the textile of Fig.~\ref{fig:zenkexample}, we find the incidence factors of crossing $2$
\begin{align*}
[2, a_1] &= 0 \\
[2, a_2] &= q+py \\
[2, a_3] &= -1
\end{align*}

Again, we may calculate, for example, $[2, a_3]$ directly from the code, noting, for example that the substring $3^+2h_1^+12^+$ representing arc $a_3$ contains the overcrossing symbol $2$ in its sub-arc of homology degree $(0,0)$ and terminates with the symbol $2^+$.
\smallskip

The full matrix from Definition~\ref{dfn:zenkpoly} for the textile in Fig.~\ref{fig:zenkexample} is
\[
A(D) = \begin{pmatrix} t & y(1-t) & -1\\ 0 & q+py & -1 \\ px & -1 & q\end{pmatrix},
\]
The determinant of this polynomial is calculated directly as:
\[
(p^2-p+pt)xy +pqx + pqty +q^2t - t
\]
from which we may calculate via appropriate substitution of equivalent polynomials and division by $t$ the reduced form of the Zenkina polynomial
\[
(q^2-p^2)xy - pqy - qx + (1-q^2).
\]

The example above demonstrates that the Zenkina polynomial can be computed directly from any realizable textile code.

\begin{thm}
\label{thm:completelist}
The final column of Table~\ref{tab:redcodes} gives a complete classification of oriented proper single component textiles up to complexity 5 modulo isotopies in a thickened torus $T^2\times I$. 
\end{thm}

\begin{proof}
Tables~\ref{tab:zenkploy22},~\ref{tab:zenkploy2} and~\ref{tab:zenkploy3} list the Zenkina polynomials for the oriented knots in a thickened torus $T^2\times I$ arising from all of the reduced textile codes enumerated in column 5 of Table~\ref{tab:redcodes}.
\smallskip

All these polynomials are distinct, and therefore arise from knots that are not isotopic to each other. 
Only for brevity, Figs.~\ref{fig:211named}, \ref{fig:221named} and~\ref{fig:311named} show the unoriented versions of these knots labelled with symbols for each distinct orientation of the diagram under the labelling from the beginning of section~\ref{sec:classification}.
\end{proof}

\begin{figure*}[!t]
\centering
\includegraphics[width=75mm]{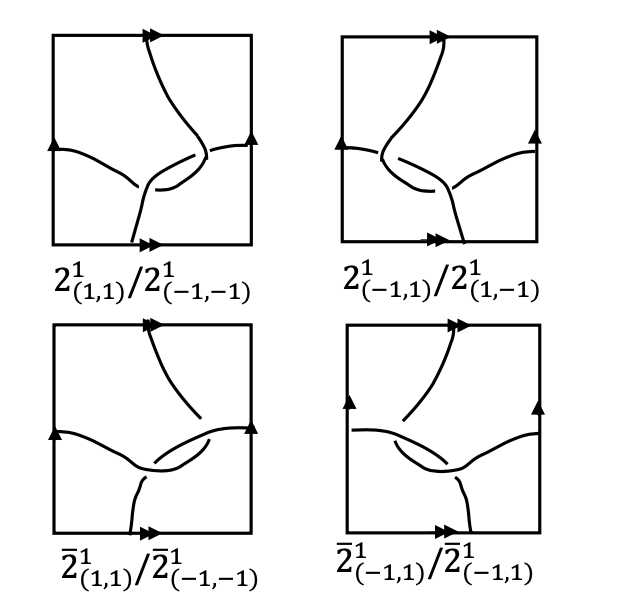}
\caption{Unoriented diagrams and knot symbols represented by all codes of complexity $4$ with $2$ crossings}  
\label{fig:211named}
\end{figure*}

\begin{figure*}[!t]
\centering
\includegraphics[width=100mm]{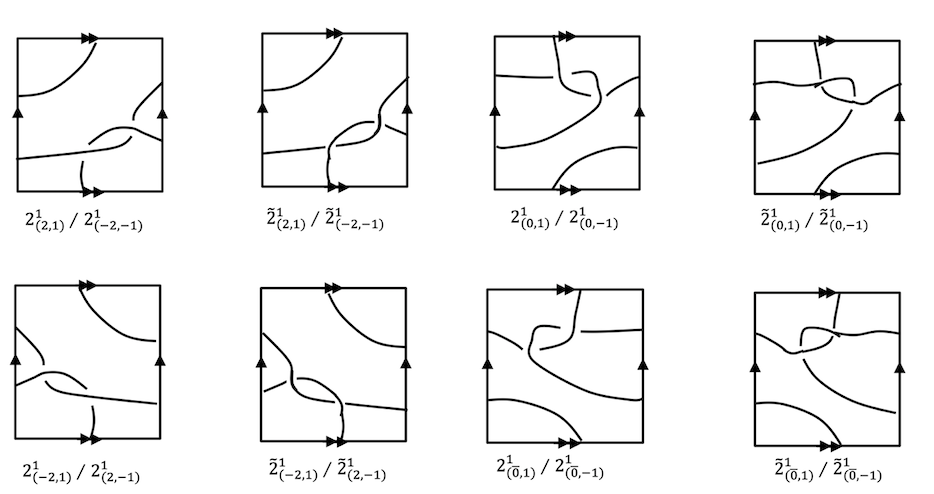}
\caption{Unoriented diagrams and knot symbols represented by all codes of complexity $5$ with $2$ crossings}  
\label{fig:221named}
\end{figure*}

\begin{figure*}[!t]
\centering
\includegraphics[width=75mm]{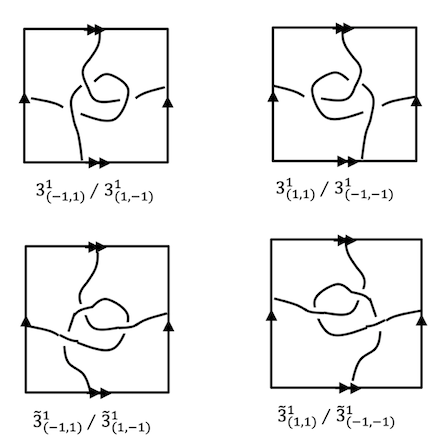}
\caption{Unoriented diagrams and knot symbols represented by all codes of complexity $5$ with $3$ crossings}  
\label{fig:311named}
\end{figure*}

\begin{table*}
\centering
\begin{tabular}{ |c|c|c| }
\hline
Paragraph & Knot symbol & Zenkina polynomial \\
\hline
$h_1^+1^+2v_1^+12^+$ & $2^1_{(1,1)}$ & $p^2xy + pqy + qx - 1$\\ 
\hline
$h_1^+12^+v_1^-1^+2$ & $2^1_{-1,1}$ & $p^2x^{-1}y + pqy + qx^{-1} - 1$ \\
\hline
$h_1^- 12^+v_1^+1^+2$ & $2^1_{1,-1}$ & $p^2xy^{-1} + pqy^{-1} + qx -1 $ \\
\hline
$h_1^-1^+2v_1^-12^+$ & $\tilde{2}^1_{(-1,-1)}$ & $p^2x^{-1}y^{-1}  + pqy^{-1}  + qx^{-1} - 1 $ \\
\hline
$h_1^+12^-v_1^+1^-2$ & $\tilde{2}^1_{(1,1)}$ & $p^2+ pqy + qx - xy$ \\
\hline
$h_1^+1^-2v_1^-12^-$ & $\tilde{2}^1_{(-1, 1}$ & $p^2 + pqy + qx - x^{-1}y$ \\
\hline
$h_1^-1^-2v_1^+12^-$ & $\tilde{2}^1_{(1,-1)}$ &  $p^2 + pqy^{-1} + qx - xy^{-1}$ \\
\hline
$h_1^-12^-v_1^-1^-2$ & $2^1_{(-1, -1)}$ & $ p^2 + pqy^{-1} + qx^{-1} - x^{-1}y^{-1}$ \\
\hline
\end{tabular}
\caption{Zenkina polynomials for knots in a thickened torus $T^2\times I$ of complexity $4$ with $2$ crossings}
\label{tab:zenkploy22}
\end{table*}

\begin{table*}
\centering
\begin{tabular}{ |c|c|c| }
\hline
Paragraph & Knot symbol & Zenkina polynomial \\
\hline
$h_1^+1+2v_1^+12+v_2^+$ & $2^1_{(2,1)}$ & $p^2x^2y + pxy + q(xy-x) - 1$ \\
\hline
$h_1^-v_2^-1+2v_1^-12+$ & $2^1_{(-2,-1)}$ & $p^2x^{-2}y + px^{-1}y^{-1} + q(x^{-1}y^{-1} - x) - 1$ \\
\hline
$h_1^-v_2^+1^-2v_1^+12^-$ & $2^1_{(2,-1)}$ & $p^2 - pqx - qxy^{-1} + x^2y^{-1}$ \\
\hline
$h_1^+1^-2v_1^-12^-v_2^-$ & $2^1_{-2,1)}$ & $p^2 - pqx^{-1} - qx^{-1}y + x^{-2}y$ \\
\hline
\hline
$h_1^+v_1^+12^+v_2^-1^+$ & $2^1_{(0,1)}$ & $p^2y + qxy + pqx^{-1} - 1$ \\
\hline
$h_1^-12^+v_2+1^+2v_1^-$ & $2^1_{(0,-1)}$ & $p^2y^{-1} - qxy - pqx ^{-1}  -y$ \\
\hline
$h_1^+v_1^-12^-v_2^+1^-2$ & $2^1_{(\bar{0},1)}$ & $p^2 + pqx^{-1}y + qx - y$\\
\hline
$h_1^-12^-v_2^-1^-2v_1^+$ & $2^1_{(\bar{0},-1)}$ & $p^2 + pqxy^{-1} + qx^{-1} - y^{-1}$ \\
\hline
\hline
$h_1^+12^-v_1^+1^-2v_2^+$ & $\tilde{2}^1_{(2,1)}$ & $p^2 - pqxy - qx ^{-1}  -x^2y$ \\
\hline
$h_1^-v_2^-12^-v1^-1^-2$ & $\tilde{2}^1_{(-2,-1)}$ & $p^2 - pqx - qxy^{-1} + x^{-2}y^{-1}$ \\
\hline
$h_1^-v_2^+12^+v_1^+1^+2$ & $\tilde{2}^1_{(2,-1)}$ & $p^2x^2y^{-1} + pqx + qxy^{-1} - 1$ \\
\hline
$h_1^+12^+v_1^-1^+2v_2^-$ & $\tilde{2}^1_{-2,1)}$ & $p^2x^{-2}y + px^{-1} + q(x^{-1}y - x^{-1}) - 1$ \\
\hline
\hline
$h_1^+v_1^+1^-2v_2^-12^-$ & $\tilde{2}^1_{(0,1)}$ & $p^2 + pqx^{-1}y + qx - y$ \\
\hline
$h_1^-1^-2v_2^+12^-v_1^-$ & $\tilde{2}^1_{(0,-1)}$ & $p^2 + pqx +qx^{-1}y^{-1} -y^{-1}$\\
\hline
$h_1^+v_1^-1^+2v_2^+12^+$ & $\tilde{2}^1_{(\bar{0},1)}$ & $p^2y + pqxy^{-1} +qx -1$ \\
\hline
$h_1^-1^+2v_2^-12^+v_1^+$ & $2^1_{(\bar{0},-1)}$ & $p^2y^{-1} + qtxy^{-1} + qx^{-1} - 1$ \\
\hline
\end{tabular}
\caption{Zenkina polynomials for knots in a thickened torus $T^2\times I$ of complexity $5$ with $2$ crossings}
\label{tab:zenkploy2}
\end{table*}

\begin{table*}
\centering
\begin{tabular}{ |c|c|c| }
\hline
Code & Knot symbol & Zenkina polynomial \\
\hline
$h_1^+12^-31^+v_1^-3^-2$ & $3^1_{(-1,1)}$ & $(1-q^2)x^{-1}y - pqx^{-1} - qy - (q^2 - p^2)$\\
\hline
$h_1^-12^-v_1^+3^+21^-3$ & $3^1_{(1,-1)}$ & $(1-q^2)x^{-1}y^{-1} - pqy^{-1} - qx^{-1} + (q^2-p^2)$\\
\hline
$h_1^+12^+31^-v_1^+3^+2$ & $3^1_{(1,1)}$ & $(q^2 - p^2)xy - pqy - qx + (1-q^2)$ \\
\hline
$h_1^-12^+v_1^-3^-21^+3$ & $3^1_{(-1,-1)}$ & $(q^2 - p^2)x^{-1}y^{-1} - pqy^{-1} - qx^{-1} + (1-q^2)$\\
\hline
\hline
$h_1^+1^-23^+1v_1^-32^+$ & $\tilde{3}^1_{(-1,1)}$ & $(q^2 - p^2)x^{-1}y - pqx^{-1} - qy + (1-q^2)$\\
\hline
$h_1^-1^+2v_1^+32^+13^-$ & $\tilde{3}^1_{(1,-1)}$ & $(q^2 - p^2)xy^{-1} - pqy^{-1} - qx + (1-q^2)$\\
\hline
$h_1^+1^+23^-1v_1^+32^-$ & $\tilde{3}^1_{(1,1)}$ & $(1-q^2)xy - pqy - qx - (q^2 - p^2)$\\
\hline
$h_1^-1^-2v_1^-32^-13^+$ & $\tilde{3}^1_{(-1,-1)}$ & $(1-q^2)x^{-1}y^{-1} - qty^{-1} - qx^{-1} + (q^2-p^2)$\\
\hline
\end{tabular}
\caption{Zenkina polynomials for knots in a thickened torus $T^2\times I$ of complexity $5$ with 3 crossings}
\label{tab:zenkploy3}
\end{table*}

The calculation of the Zenkina polynomial and other invariants directly from a textile code is in progress. 
Exploring the properties of these invariants may reveal more effective approaches to realizability of codes and classifications of textiles of greater complexity and with multiple components. 

\section{Conclusions and a discussion of future work}
\label{sec:discussion}

Here are the key contributions to the shape modelling and topological classifications of periodic textile structures.
\smallskip

\noindent
$\bullet$
Definition~\ref{dfn:textile_code} introduces a 1-dimensional string to encode any textile structure, which allows a simple reconstruction. 
\smallskip

\noindent
$\bullet$
Theorem~\ref{thm:algorithm} justifies a linear time algorithm to detect realizability of any textile code by an actual textile structure.
\smallskip

\noindent
$\bullet$
The systematic approach to the enumeration of textile structures by new codes has led to the isotopic classification of all oriented knots in $T^2\times I$ up to complexity 5 by Theorem~\ref{thm:completelist}.
\smallskip

The implementation of Reidemeister moves directly from a textile code described in Fig.~\ref{fig:reidcodes} is not exhaustive - there may other more complex reductions that can be similarly handled. 
\smallskip

Determination of improper textiles in the sense of Definition~\ref{dfn:textile} is  manual - we can explore ways to determine whether a link is proper or otherwise directly from the code. 
\smallskip

Another direction is to characterize fundamental groups of 2-periodic link complements in $T^2\times I$ similarly to classical links \cite{kurlin2007peripherally}.
This work was supported by the UK Engineering and Physical Sciences Research Council under grant  EP/R018472/1.
We thank all reviewers for helpful suggestions.


\bibliographystyle{plain}
\bibliography{realizability_textiles}

\end{document}